\documentclass[journal]{IEEEtran}
\IEEEoverridecommandlockouts                              

\IEEEoverridecommandlockouts
\usepackage{stfloats}
\usepackage{bbm}

\usepackage{graphicx} 
\usepackage{amsmath} 
\usepackage{amssymb}  

\usepackage{lipsum}
\usepackage{graphicx}
\usepackage[T1]{fontenc}
\usepackage{aecompl}
\usepackage{amsfonts}
\usepackage{bbm}
\usepackage{ifthen}
\usepackage{subcaption}
\usepackage{multirow}
\usepackage{diagbox}
\usepackage{ifthen}
\usepackage{cite}
\usepackage[usenames, dvipsnames]{color}
\usepackage{url}
\usepackage{hyperref}
\usepackage[hyphenbreaks]{breakurl}
\usepackage{mathtools}	
\usepackage[normalem]{ulem}
\usepackage[noend]{algorithm2e}
\RestyleAlgo{ruled}
\usepackage{algpseudocode}
\usepackage{tabularx}
\usepackage{hhline}
\usepackage{xfrac}

\usepackage{amsthm}
\usepackage{balance}
\usepackage{mathtools}
\mathtoolsset{showonlyrefs=true}

\newtheorem{theorem}{Theorem}
\newtheorem{lemma}{Lemma}

\newtheorem{defn}{Definition}
\newtheorem{corollary}{Corollary}

\newtheorem{remark}{Remark}
\newtheorem{assumption}{Assumption}

\DeclareMathOperator{\vol}{vol}
\newcommand{\norm}[1]{\left\lVert#1\right\rVert}

\usepackage{bbm}

\newcommand{\A}{\mathcal{A}}

\newcommand{\R}{\mathcal{R}}
\DeclareMathOperator{\opint}{int}

\makeatletter
\newcommand*{\centerfloat}{%
  \parindent \z@
  \leftskip \z@ \@plus 1fil \@minus \textwidth
  \rightskip\leftskip
  \parfillskip \z@skip}
\makeatother

\usepackage{enumerate}
\usepackage[shortlabels]{enumitem}

\usepackage{ifthen}
\newboolean{showcomments}
\setboolean{showcomments}{true}
\usepackage[usenames,dvipsnames]{color}
\usepackage{todonotes}

\definecolor{bleudefrance}{rgb}{0.19, 0.55, 0.91}
\definecolor{ao(english)}{rgb}{0.0, 0.5, 0.0}

\newcommand{\addcite}[0]{\ifthenelse{\boolean{showcomments}}
{\textcolor{purple}{(add cite(s)) }}{}}%

\newcommand{\enrique}[1]{  \ifthenelse{\boolean{showcomments}}
{\todo[inline,color=bleudefrance, caption ={}]{Enrique: #1}}{}}
\newcommand{\emmargin}[1]{\ifthenelse{\boolean{showcomments}}{\marginpar{\color{bleudefrance}\tiny EM: #1}}{}}

\newboolean{showedits}
\setboolean{showedits}{false}
\usepackage[markup=underlined]{changes}
\definechangesauthor[color=bleudefrance]{EM}
\newcommand{\aem}[1]{
\ifthenelse{\boolean{showedits}}
{\added[id=EM]{#1}}
{\!#1\hspace{-4.75pt}}
}
\newcommand{\repem}[2]{
\ifthenelse{\boolean{showedits}}
{\replaced[id=EM]{#1}{#2}}
{\!#1\hspace{-4.75pt}}
}
\newcommand{\dem}[1]{
\ifthenelse{\boolean{showedits}}
{\deleted[id=EM]{#1}}
{}
}

\makeatletter
\if@todonotes@disabled

\else

\fi
\makeatother

\newcommand{\hl}[1]{\textcolor{black}{#1}}
\newcommand{\mypara}[1]{\smallskip\noindent\textbf{#1.}\ }
\newcommand{\red}[1]{#1}

\newcommand{\blue}[1]{#1}

\newboolean{arxiv}
\setboolean{arxiv}{true}
\newboolean{lemmas}
\setboolean{lemmas}{true}
\allowdisplaybreaks

\usepackage{setspace}

\newboolean{with-appendix}
\setboolean{with-appendix}{false}
\mathtoolsset{showonlyrefs=true}

\title{A Recurrence-based Lyapunov Direct Method for Stability Analysis}
\title{A Recurrence-based Direct Method for Stability Analysis and GPU-based Verification of Non-monotonic Lyapunov Functions}
\title{Stability Analysis and Data-driven Verification via\\ Recurrent Lyapunov Functions}

\author{Roy Siegelmann, Fernando Paganini, and Enrique Mallada%
\thanks{Manuscript received XXX; revised XXX; accepted XXX. This work was supported by NSF through Grants CAREER 1752362, CPS 2136324, and Global Centers 2330450, by the DOE Office of Science (ASCR) under Award No.~826565, and by AFOSR under Grant FA9550-23-1-0350. (Corresponding author: Enrique Mallada.)}%
\thanks{R. Siegelmann is with the Massachusetts Institute of Technology (MIT), Cambridge, MA, USA (e-mail: resiege@mit.edu).}%
\thanks{F. Paganini is with Universidad ORT Uruguay, Montevideo, Uruguay (e-mail: paganini@ort.edu.uy).}%
\thanks{E. Mallada is with Johns Hopkins University, Baltimore, MD, USA (e-mail: mallada@jhu.edu).}}

\begin{document}
\bstctlcite{MyBSTcontrol} 

\allowdisplaybreaks

\maketitle
\begin{abstract}
Lyapunov's direct method is a cornerstone of stability and control, but it hinges on finding a Lyapunov function, a task demanding ingenuity or computation. A key difficulty is that every sub-level set of the function must be forward invariant, coupling its geometry to the system's trajectories. We relax this by replacing invariance with recurrence: a set is recurrent if every trajectory starting in it returns within a finite time. This yields the notion of a Recurrent Lyapunov Function (RLF), whose sub-level sets need only be recurrent. We show that, under mild conditions, RLFs guarantee stability, and we introduce stronger notions yielding asymptotic and exponential stability. We also give norm-based converse theorems: under the corresponding stability conditions, any norm is an RLF for their practical versions. We then develop GPU-based algorithms that certify (practical) stability from trajectory data alone, without a Lyapunov function. Certifying stability up to an $\varepsilon$-neighborhood needs only $O(\log(1/\varepsilon))$ trajectory evaluations, with constants growing as the certified decay rate nears the true one, exposing an intrinsic performance-cost trade-off.
\end{abstract}

\medskip

\section{Introduction}\label{ct-Intro}

Lyapunov stability theory plays a central role in the study of dynamical systems. It provides a rigorous mathematical framework for qualitatively analyzing system solutions and has heavily influenced systems theory and engineering over the past century. 
\red{Its fundamental tool, Lyapunov's direct (second) method~\cite{lyapunov1992general}, states} mild conditions on a function $V(x)$ (non-increasing along trajectories and proper) that can certify stability of an equilibrium point. \red{Since its inception in 1892, the method has found ubiquitous applications across engineering, e.g., aerospace, electrical, mechanical, and chemical~\cite{sontag2013mathematical,parks1992lyapunov,sastry1999lyapunov,Khalil2002}.} 

\red{A critical step in applying Lyapunov's direct method is finding a function $V$ that satisfies the required conditions.} \red{Unfortunately, while such a function is known to exist via converse theorems, e.g.,~\cite{massera1949liapounoff}, manually finding one is often tricky, relying on ingenuity and domain knowledge.} To circumvent this step, a variety of computational methods have been proposed for finding Lyapunov functions \cite{Peter2015}, e.g., via the use of partial differential equation (PDE) solvers to solve Zubov's Equation \cite{HASSAN1981,VANNELLI198569}, linear programs (LPs) to find piece-wise linear Lyapunov functions~\cite{Pedro1999}, and semidefinite programs (SDPs) to solve linear matrix inequalities (LMIs) \cite{Goebel2006} or sum-of-squares (SOS) problems \cite{papachristodoulou2002sos}.
\red{However, the computational complexity increases exponentially with both the state dimension and the Lyapunov function parameterization~\cite{Peter2015,ahmadi2019dsos}.}

This has led to multiple investigations  into relaxing the conditions required for $V$, and in particular, its time derivative $\dot V$.
Such relaxations can be broadly divided into three groups. The first group seeks LaSalle-Krasovskii type of conditions by relaxing the negative definiteness of $\dot V$, i.e., only requiring $\dot V\leq 0$; see~\cite{rouche1977stability,Khalil2002} and its generalization~\cite{mazenc2004strong,malisoff2009constructions}. \red{The second group further relaxes negative definiteness by allowing} $\dot V>0$ on some regions of the state space. This is implicitly done by using 
generalizations of the comparison lemma~\cite{gunderson1971comparision} to impose conditions on higher order time derivatives of $V$ that still ensure convergence of \red{$V\rightarrow 0$ while allowing $\dot V>0$.} The third group uses the so called discretization method, which considers a fixed parameter $T>0$ and leverages the net decrement of $V$ across any trajectory $x(t)$, i.e., $V(x(t+T)) - V(x(t))$, to reason about stability~\cite{aeyels1998new,coron1994relation}.
Unfortunately, despite such efforts, the basic principle can still be traced back to the (indirect) construction of a Lyapunov function whose sub-level sets are invariant~\cite{ahmadi2008non,karafyllis2011can}, \red{which still must be verified analytically or via a convex program, rendering similar challenges.}



\red{The crux is that Lyapunov's direct method implicitly constrains their geometry by requiring every sub-level set to be positively invariant.} 
\red{This tightly couples level-set geometry to the vector field, making such functions difficult to construct.} 
\red{This paper relaxes this condition by replacing invariance of sub-level sets with the weaker notion of recurrence.} We say that a set is ($\tau$-) recurrent if every trajectory that starts in the set returns to it (within $\tau$ seconds). 
\red{This decouples level-set geometry from instantaneous vector-field alignment while still enabling stability guarantees.} 
\red{Recurrence has recently proved a versatile mechanism for estimating regions of attraction} of stable equilibria~\cite{shen2022model} and for verifying safety properties of dynamical systems~\cite{ssm2024allerton}.
\red{Moreover, (control) recurrence can be achieved at lower data rates than invariance~\cite{sm2024hscc} and often enforced from finitely many trajectory samples~\cite{sibai2026recurrence}.}

In this paper, we seek to explore the role of recurrence in certifying different notions of stability of an equilibrium point. The contributions of our work are several:
\begin{itemize}[leftmargin=*]
    \item\emph{Recurrent Lyapunov Functions:}
    \red{We introduce Recurrent Lyapunov Functions (RLFs, Definition~\ref{defn:RLF}), which generalize classical Lyapunov functions by replacing invariance of sub-level sets with recurrence.}
    \hl{This relaxation decouples the geometry of the level sets from the vector field.}

    \item\emph{Stability Guarantees:} 
    \red{We establish stability theorems demonstrating that} \blue{ existence of an RLF over a compact set}  \red{guarantees stability, asymptotic stability, and exponential stability} \hl{ (Theorems~\ref{thm:recurrence-stability},~\ref{thm:asymptotic-recurrence-stability},~\ref{thm:tau-exponential-stability})}. 
    \red{These results enable stability proofs without strict invariance conditions.}

    \item\emph{Norm-agnostic Converse Theorems:}
    \hl{We show that any norm is guaranteed to satisfy} \blue{a \emph{practical} ($\varepsilon$-relaxed) version of} \hl{our RLF conditions (Theorems~\ref{thm:norm-srlf} and~\ref{thm:norm-erlf}), provided} \blue{the system is asymptotically or exponentially stable; these relaxed conditions in turn certify practical stability.}
    \red{This highlights the fundamental role of recurrence in stability analysis and opens the door to verification methods that do not require computing a Lyapunov function.}
    \item\hl{\emph{Data-Driven Verification with Guarantees:}
    We develop two GPU-parallelized algorithms for trajectory-based verification of the RLF conditions: one for the best decay rate over a fixed region (Algorithm~\ref{alg:region-verification}), the other for the largest region of attraction at a target rate $\alpha$ (Algorithm~\ref{alg:grow-alpha-roa}). We further show that $O(\log(R/\varepsilon))$ trajectory evaluations suffice (Theorem~\ref{thm:sample-complexity}) to certify practical exponential stability on $B_R\!\setminus\! B_\varepsilon$, with constants exposing an intrinsic trade-off between sample complexity and certified performance.}
\end{itemize}


\hl{\red{A preliminary version appeared in \cite{sspm2023cdc}; this paper extends it in several ways.} First, we extend our stability analysis to Recurrent Lyapunov Function conditions over arbitrary sets containing the equilibrium, rather than sub-level sets. Second, we introduce novel converse theorems showing that any norm can satisfy suitable weak versions of these 
conditions; these in turn are shown to impose \emph{practical} asymptotic and exponential stability. Third, we  provide estimates on the sample complexity of verifying practical exponential stability over bounded regions of the state space. Finally, we develop a  verification algorithm and provide numerical validations illustrating the merits of our framework.}


\paragraph*{Closely related work} The derived conditions are similar in spirit to the ones considered by Karafyllis in \cite{karafyllis2011can}, which studies robust stability analogs (cf. Propositions 2.3 and 2.5). Particularly, our asymptotic stability condition is closely related to~\cite[Prop. 2.5]{karafyllis2011can}.
Our stability, \hl{ asymptotic stability,} exponential stability \red{conditions}, and \hl{their practical counterparts} are, however, new and not present in prior work.
More importantly, the focus of our paper is on exploring the connection of such conditions with the recurrence of level sets of $V$ and developing parallelizable algorithms that can be implemented on GPUs, whereas \cite{karafyllis2011can} focuses on robust stability and provides Matrosov-type conditions.
\hl{At the data-driven verification level, the learning-based method of Boffi et al.~\cite{boffi2021learning} has a sample complexity of $\Omega(\varepsilon^{-2d})$ in the resolution $\varepsilon$, under an incremental-stability assumption on the underlying system. Our $O(M^d\log(R/\varepsilon))$ bound (Theorem~\ref{thm:sample-complexity}) leverages recurrence in lieu of incremental stability: the $\varepsilon$ dependence drops from polynomial to logarithmic, while the constant $M$ encodes the gap $(\lambda - \alpha)$ between the system's exponential rate $\lambda$ and the certified rate $\alpha$, making the trade-off between certified performance and sample complexity explicit.}

\hl{The paper is organized as follows. Section~\ref{ct-prelim} establishes preliminaries on dynamical systems and stability; Section~\ref{ct-recur} introduces recurrent sets and their use in bounding trajectories. Sections~\ref{ct-rlf}--\ref{ct-exponential} develop Recurrent Lyapunov Functions and the corresponding stability, asymptotic stability, and exponential stability theorems. Section~\ref{ct-weak converse} shows that norms satisfy the RLF conditions, yielding practical stability guarantees and norm-agnostic converse theorems. Section~\ref{ct-verify} develops trajectory-based verification algorithms, Section~\ref{ct-numerical} presents numerical experiments, and Section~\ref{ct-conclusions} concludes.}


\textit{Notation:} Throughout the text, we let \( \|\cdot\| \) denote an arbitrary norm on \( \red{\mathbb{R}^d} \), and define \( B_r(x) \) as the closed ball of radius \( r \) centered at \( x \in \red{\mathbb{R}^d} \). Given a set \( S \subset \red{\mathbb{R}^d} \), the distance from a point \( y \in \red{\mathbb{R}^d} \) to \( S \) is defined as
$\mathrm{d}(y, S) := \inf_{x \in S} \|y - x\|.$
We also use \( \mathbb{R}_{\geq 0} := \{ x \in \mathbb{R} \mid x \geq 0 \} \), \( \mathbb{R}_{> 0} := \{ x \in \mathbb{R} \mid x > 0 \} \), and $[n]:=\{1,\dots,n\}$.
\section{Preliminaries}\label{ct-prelim}
 We consider a continuous-time dynamical system 
\begin{equation}\label{eq:system}
 \dot{x} = f(x)\,,  
\end{equation}
where $x\in D\subseteq {\mathbb{R}^d}$ is the state, and the map $f: D \rightarrow {\mathbb{R}^d}$ is a continuous function defined over an \textit{open domain} $D$. 
Given an initial state $x$, we use $\phi(t,x)$ to denote the solution of \eqref{eq:system}. Throughout the paper, we make the following assumption about the vector field and its solutions.

\begin{assumption}\label{as:Lipschitz}
    The vector field $f(x)$ in \eqref{eq:system} is locally Lipschitz. That is, for any compact set $S \subset D$, there exists a constant $\overline{L}_S\in\mathbb{R}_{\geq0}$ such that
     \[
     \|f(y)-f(x)\|\leq \overline{L}_S \|y-x\|, \qquad \forall x,y \in S.
     \] 
\end{assumption}

The local Lipschitz nature of the vector field implies that solutions must exist for some amount of time, which we will denote by the following:

\begin{defn}[Interval of Existence]
\label{defn:IoE}
For $x \in D$, the \textbf{maximal interval of existence} $I(x) \subset \mathbb{R}$ is the largest open interval around $t=0$ such that $\phi(t,x)$ exists for all $t \in I(x)$. The trajectory is said to be \textbf{forward complete} if $I(x) \supset [0,\infty)$.
\end{defn}

Whenever the initial condition is understood from the context, we will use {$x(t):=\phi(t,x)$}. 

It will be useful to introduce \emph{one-sided} Lipschitz bounds. 
This concept, traditionally stated in inner-product spaces, has recently  been extended \cite{davydov2022non} to a larger family of norms in ${\mathbb{R}^d}$.  The construction is based on concept of \emph{weak-pairing} $[\cdot;\cdot]:{\mathbb{R}^d}\times{\mathbb{R}^d} \to \mathbb{R}$,  a  generalized inner product satisfying
:
\begin{itemize}
\item[(i)] 
$[x_1 + x_2;y]\leq [x_1;y]+ [x_2;y]$ (subadditivity in first component); 
\item[(ii)]
$[\alpha x;y]=[x; \alpha y] = \alpha [x;y]$ for $\alpha>0$; $[-x;-y]=[x;y]$ 
(weak homogeneity); 
\item[(iii)] $[x;x]>0$ for $x\neq 0$ (positive definiteness);
\item[(iv)] 
$|[x;y]| \leq [x;x]^\frac{1}{2} [y;y]^\frac{1}{2}$ (Cauchy-Schwarz inequality).
\end{itemize}

For every norm there exists a weak pairing satisfying 
$\|x\| =[x;x]^\frac{1}{2}$; for the (weighted) Euclidean norm the natural pairing is the standard (weighted) inner product. In \cite{davydov2022non},
weak pairings are provided as well for (weighted) $l_1$, $l_\infty$ norms in ${\mathbb{R}^d}$. These are also shown to satisfy some additional properties, in particular we will use the \emph{curve norm derivative formula}
\begin{equation} \label{eq:curve norm}
\|x(t)\|\cdot D^+\|x(t)\| = [\dot{x}(t);x(t)],
\end{equation}
where $D^+\varphi(t) := \limsup_{h\to 0+} \frac{\varphi(t+h)-\varphi(t)}{h}$ denotes the upper-right Dini derivative.\\
The following definition is also based on \cite{davydov2022non}: 
\begin{defn}[One-sided Lipschitz] \label{def:one-side-L}
For system \eqref{eq:system} under Assumption \ref{as:Lipschitz}, the one-sided Lipschitz constant over the compact set $S\subset D$ corresponding to a norm $\|\cdot\|$ is defined as the smallest $L_S\in \mathbb{R}$ such that  
\[
 [f(y)-f(x);y-x] \leq L_S\|y-x\|^2, \quad \forall x, y \in S, \]
 where $[\cdot;\cdot]$ is a weak pairing associated with $\|\cdot\|$. \\
It follows from 
condition (iv) above that $L_S \leq  \overline{L}_S$. 
\end{defn}
For a continuously differentiable field $f(x)$ and a \emph{convex} domain $S$, 
it is shown in \cite{davydov2022non} that the one-sided Lipschitz constant can be computed from the Jacobian matrix:
\begin{equation}\label{eq:mujacobian}
    L_S=\sup_{x\in S}\mu\left(\tfrac{\partial f}{\partial x}(x)\right).
\end{equation}
Here $\mu(A)=\lim_{h\to 0+} (\|I+hA\|-1)/h$, the logarithmic matrix norm associated with the vector norm under consideration. In particular, for $\|\cdot\|_\infty$ and $\|\cdot\|_2$, $\mu_\infty(A)=\max_i\big(a_{ii}+\sum_{j\neq i}|a_{ij}|\big)$ and $\mu_2(A)=\lambda_{\max}\!\big(\tfrac{A+A^\top}{2}\big)$, respectively. \\

We next review the core building blocks of Lyapunov Stability Theory.
\begin{defn}[Stability] \label{defn:stability}
    An equilibrium $x^*$ is \textbf{stable} if for any $\varepsilon>0$, $\exists \delta >0$, such that if $\|x-x^*\|\leq \delta$ then $\|\phi(t,x)-x^*\|\leq \varepsilon$ $\forall t\geq0$.
\end{defn}

\begin{defn}[Attractivity]
    An equilibrium $x^*$ is \textbf{attractive on the set $S$} if for every $x\in S$, $\|\phi(t,x)-x^*\|\rightarrow 0 $ as $t\rightarrow\infty$. 
\end{defn}

\begin{defn}[Asymptotic Stability]\label{defn:asymptotic stability}
    An equilibrium $x^*$ is \textbf{asymptotically stable on the set $S$}, where $x^*\in \mathrm{int}(S)$, if it is stable, and attractive on $S$.
\end{defn}

\begin{defn}[Exponential Stability]\label{defn:exp stability}
    An equilibrium $x^*$ is \textbf{exponentially stable} on the set $S$, where $x^*\in \mathrm{int}(S)$,   if there exist constants $C\ge 1$, $\lambda>0$ such that if $x\in S$,  then 
    \begin{equation}\label{eq:lambda-exp}
        \|\phi(t,x)-x^*\|\leq Ce^{-\lambda t}\|x - x^{*}\|,\;\;\;  \forall t\geq0.
    \end{equation}
\end{defn}
It will also be useful to define sets that are of general use to characterize transient, as well as asymptotic behavior. 

\begin{defn} [Reachable Tube]\label{defn:reachable-set}  
For the dynamical system \eqref{eq:system}, a time $\tau > 0$, and a set $S \subset D$, we denote the $\tau$-reachable tube from $S$ within $\tau$ units of time by 
$$\mathcal{R}^{\tau}(S) = \bigcup_{x \in S, t \in [0, \tau]\cap I(x)} \{\phi(t,x)\}.$$
\end{defn}


\begin{defn}[Positively Invariant Sets]\label{defn:invariant set}
A set $S\subseteq {\mathbb{R}^d}$ is positively invariant w.r.t. \eqref{eq:system} if and only if:
\begin{equation}\label{eq:invariant-criterion}
x\in S\implies \phi(t,x)\in S,\quad \forall\, t\in \mathbb{R}_{\geq 0}.
\end{equation}
\end{defn}
Since we only consider here positively invariant sets, as opposed to negatively invariant sets, we will often refer to them as plainly invariant sets.
As mentioned before,  the notion of positive invariance is a fundamental building block of Lyapunov Theory. By trapping trajectories on compact sub-level sets of a function one can guarantee boundedness of trajectories, stability, and asymptotic stability via a gradual reduction of the Lyapunov function value. 
\section{Recurrence}\label{ct-recur}

To relax the notion of invariance, one must allow trajectories to temporarily leave a set. However, in order to still be able to make statements about asymptotic behavior, our first condition requires trajectories to always come back.

\begin{defn}[Recurrent Set]\label{defn:recurrent}
A set $S \subseteq {\mathbb{R}^d}$ is recurrent w.r.t. \eqref{eq:system}, if for any $x\in S\,$, and  $t\geq0$,
\begin{equation}\label{eq:recurrence}
    \exists\; t' > t,\quad  \text{s.t.}\quad \phi(t',x)\in S.
\end{equation}
\end{defn}

Since trajectories are allowed to leave $S$, in our development, it will be useful to keep track of the  time intervals where a trajectory $\phi(t,x)$ lies within a given set $S$ for a given initial point $x\in D$.
\hl{\begin{defn}[Containment Times]
    Given a set $S\subset D$, a point $x\in D$, and a horizon $\tau>0$, we define
    $$T_S(x;\tau) := \{t \in (0,\tau] \mid \phi(t,x) \in S\},$$
    the set of \emph{containment times} of $x$ in $S$ over the interval $(0,\tau]$.
\end{defn}}

The notion of recurrent sets introduced here is related to classical Poincar\'e recurrence~\cite{poincare1893methodes}, and in particular, Poincar\'e recurrent sets~\cite[Def. 2.4.1]{alongi2007recurrence}, which constitutes the union of Poincare recurrent points; a point $x$ is Poincare recurrent if its backward and forward flows, i.e., $\{\phi(-t,x)\}_{t\geq0}$ and $\{\phi(t,x)\}_{t\geq0}$, get arbitrarily close to $x$, \textit{infinitely often}. 
In fact, one can show that any open subset $S$ of a Poincar\'e recurrent set is a Recurrent Set according to Definition \ref{defn:recurrent}.


\dem{Implicit in Definition \ref{defn:recurrent}  is that trajectories which start in $S$ are \textit{forward complete}, since the flow is defined for arbitrarily large times. See more comments on this assumption in Remark \ref{rem:forward complete}. Another consequence of Definition \ref{defn:recurrent} is that trajectories that start in $S$ will visit it \emph{infinitely often} (again and again), and \textit{forever} (there is always a future time when it is visited again). These properties will allow us to make statements on asymptotic behavior, under appropriate assumptions on $S$. However, it does not allow to bound how far trajectories can travel between visits to $S$, limiting their application to stability analysis. This motivates the notion of $\tau$-recurrent sets.}
\aem{Definition~\ref{defn:recurrent} implicitly requires that trajectories starting in $S$ are forward complete (see Remark~\ref{rem:forward complete}) and ensures they visit $S$ infinitely often. These properties enable statements about asymptotic behavior, but do not bound how far trajectories travel between visits---motivating the stronger notion of $\tau$-recurrent sets.}

\begin{defn}[$\tau$-Recurrent Set]\label{defn:T recurrent}
A set $S\subseteq D$ is $\tau$-recurrent w.r.t. \eqref{eq:system}, if
for any $x\in S\,$ and $t\geq0$, 
\begin{equation}\label{eq:T-recurrence}
    \exists\; t'> t,\quad \text{with}\quad t'-t\in (0,\tau] \quad \text{and}\quad \phi(t',x)\in S.
\end{equation}
We further say that $S$ is \emph{strictly $\tau$-recurrent,} if for any $x \in S\,$, and $t \geq 0$,
\begin{equation}\label{eq:strictly-T-recurrence}
    \exists\; t'> t,\quad \text{with}\quad t'-t\in (0,\tau] \quad \text{s.t.}\quad \phi(t',x)\in S \backslash \partial S.
\end{equation}
\end{defn}

While Definition \ref{defn:T recurrent} is sufficient for the development that follows, its conditions for far-off times are hard to verify. The next lemma shows that such verification is simpler when the set $S$ is compact.

\begin{lemma}[Characterization of Compact $\tau$-Recurrent Sets]
\label{lem:infinite-seq}
Let $S \subset D$ be a compact set, and consider the system \eqref{eq:system} under Assumption \ref{as:Lipschitz}, and $\tau>0$. The following conditions are equivalent:
\begin{enumerate}[$(i)$]
    \item $S$ is $\tau$-recurrent.
    \item For any $x\in S$, $\exists$ $t \in(0,\tau ]\cap I(x)$ with $\phi(t,x)\in S$.
    \item For any $x\in S$ there is a sequence $\{t_n\}_{n\in \mathbb{N}}$  satisfying,
    \begin{equation}\label{eq:tau-decreasing-sequence-pre}
    {\lim_{n\rightarrow\infty}t_n = \infty\,, \quad\text{with}\quad t_{n+1}-t_n\in (0,\tau]}\,, 
    \end{equation}
    {and $\phi(t_n, x) \in S$ $\forall n$.}
\end{enumerate}

\end{lemma}
\ifthenelse{\boolean{lemmas}}{
\begin{proof} \text{ }

\noindent
     $(i)\!\!\!\implies\!\!\!(ii)$: Follows from Definition~\ref{defn:T recurrent} and choosing $t=0$.\\
     
\noindent
     $(ii)\!\!\!\implies\!\!\!(iii)$: 
     Given $x\in S$,  we build the sequence $\{t_n\}_{n\in \mathbb{N}}$ satisfying \eqref{eq:tau-decreasing-sequence-pre} and $\phi(t_n,x)\in S$ by induction.   
    For the base case, let $t_0:=0$, $x_0=\phi(t_0,x)=x$, and define:
    \[
    t_1=\max\{t \in (0, \tau] \mid \phi(t,x_0) \in S\};
    \]
    the above set of times is non-empty by hypothesis. Its supremum is actually a maximum due to the compactness of $S$ and the continuity of $\phi(t,x)$. By construction, $t_1\in I(x)$. 
    
    
    The inductive construction proceeds in a similar manner: given $t_1<t_2< \cdots t_n$, with $x_n:=\phi(t_n,x) \in S$, define:
    \begin{equation}\label{eq:t_n sequence}
    t_{n+1}=t_n + \max\{t \in (0, \tau ] \mid \phi(t,x_n) \in S\}.
    \end{equation}
    Note that $t_{n+1}$ exists analogously to the above, and satisfies $t_{n+1} - t_n \in (0,\tau]$ as required. Further, $t_{n+1}-t_n\in I(\phi(t_n,x))$, which implies that $t_{n+1}\in I(x)$. 
    
    It remains to show that $t_n\rightarrow\infty$, which we argue by contradiction. If, instead, the strictly increasing sequence of times was bounded, we would have $t_n\uparrow t^*$. If $t^*$ was the supremum of the maximal interval $I(x)$, then a standard result in differential equations \cite{coddington} implies that $\phi(t,x)$ should exit any compact set as $t\uparrow t^*$; this would contradict the fact that $x_n:=\phi(t_n,x)\in S \ \forall n$ , and $S$ compact. Therefore $[0,t^*] \subset I(x)$, $x^*=\phi(t^*,x)$ exists, and by 
    continuity of $\phi(\cdot,x)$, it follows that $\lim_{n\rightarrow\infty}x_n =x^* \in S$. 
    
    We now choose $n$ large enough, say $n^*$, such that $t^*<t_{n^*}+\tau$. 
    Since $\phi(t^*,x)\in S$, then $t^*$ is a candidate value for the induction step \eqref{eq:t_n sequence} and satisfies $t^*>t_{n^*+1}$; this contradicts the fact $t_{n^*+1}$ is the maximum such value.
    Thus, $t_n \to \infty$, as desired. In particular, $I(x)\supset [0,\infty)$.\\

    \noindent
     $(iii)\!\!\!\implies\!\!\!(i)$: 
    Given $x\in S$ and $t\geq0$, let $n^*$ be the largest $n$ s.t. $t_n\leq t$, then, it follows that $t_{n^*+1}-t\in (0,\tau]$, $\phi(t_{n^*+1},x)\in S$. 
    Definition \ref{defn:T recurrent} is satisfied with $t'=t_{n^*+1}$.
\end{proof}
}{
\noindent \textit{Proof.}~The proof is omitted due to page limits. \hfill \qed
}
\begin{remark}[Forward Completeness]\label{rem:forward complete}
As a consequence of Lemma~\ref{lem:infinite-seq}, condition $(ii)$ guarantees forward completeness of trajectories initiating in the compact set $S$. Hence, the implicit forward completeness requirement in Definition~\ref{defn:recurrent} is not restrictive when $S$ is compact.
\end{remark}

\dem{Another advantage of compact  $\tau$-recurrent sets is that they allow us to bound the distance a trajectory can travel away from it--a critical condition for stability. To show this, we recall first that the vector field \eqref{eq:system} is assumed locally Lipschitz (Assumption \ref{as:Lipschitz}). While such property suffices, it will prove convenient to obtain tighter bounds via locally one-sided Lipschitz constants.}
\aem{Another advantage of compact $\tau$-recurrent sets is the ability to bound how far trajectories can travel away from them---a critical step for stability.}

\dem{We will also introduce a notation for the maximum norm of the vector field on a (compact) set $S$,
\[
F_S:=\max_{S}\|f(x)\|.\]
When $S$ is a ball of radius $\varepsilon$ centered around a fixed point $x^*$, i.e., $S=B_\varepsilon(x^*)$, we will use $F_\varepsilon:=F_{B_\varepsilon(x^*)}$ for simplicity. Note that $F_\varepsilon\to 0$ when $\varepsilon\to0$.}
\aem{We will write $F_S := \max_{x \in S} \|f(x)\|$ for the maximum norm of the vector field on a compact set $S$, with shorthand $F_\varepsilon := F_{B_\varepsilon(x^*)}$ when $S = B_\varepsilon(x^*)$; note $F_\varepsilon \to 0$ as $\varepsilon \to 0$.}
\dem{Using these definitions, we now derive a bound on how far trajectories can go from a compact set in which they start.}

\begin{lemma}[Containment Lemma]\label{lem:containment}
    Consider system \eqref{eq:system} under Assumption \ref{as:Lipschitz}. Let $S \subset D$ be a compact set such that solutions starting in $S$ are forward complete. Then, $\mathcal{R}^{\tau}(S)$ has compact closure, 
    and defining 
    $L :=L_{\mathrm{cl}\mathcal{R}^{\tau}(S)}<\infty$, we have  
    \begin{equation}\label{eq:containment}
        \max_{t\in[0,\tau]} \mathrm{d}(\phi(t,x),S) \leq F_S \,h(\tau;L),
    \end{equation}\vspace{-2ex}
    \begin{flalign*}
        &\text{where}\qquad\quad h(\tau;L) := \begin{cases}
            \frac{e^{L\tau}-1}{L}, & L\ne0,\\
            \tau, & L=0.
        \end{cases} &
    \end{flalign*}
\end{lemma}
\begin{proof}
Compactness of $\mathrm{cl}(\mathcal{R}^{\tau}(S))$ is proved in Proposition $5.1$ of \cite{sontag_containment}.   
Consider $x_0 \in S$ and let $x(t)=\phi(t,x_0)$, $u(t):=\|x(t)-x_0\|$, with $t\in [0,\tau]$. Observe that since $x_0\in S$, $\mathrm{d}(x(t), S) \leq u(t)$. Thus, bounding $u(t)$ will be sufficient. 

We now apply formula \eqref{eq:curve norm} with 
$x(t)-x_0$ in lieu of $x(t)$:
    \begin{align*}
     u(t) D^+u(t)& = [\dot{x}(t);x(t)-x_0] 
       = [f(x);x-x_0] \\
      & \leq [f(x)-f(x_0);x-x_0] + [f(x_0);x-x_0] \\
      & \leq L \|x-x_0\|^2 + \|f(x_0)\| \|x-x_0\| \\
        & \leq Lu(t)^2 + F_S u(t),
    \end{align*}
    where the first inequality comes from the subadditivity (i) of weak pairing, the second from 
    Definition \ref{def:one-side-L} and the Cauchy-Schwarz property (iv), and the last from $x_0\in S$. 
    
If $x_0$ is an equilibrium of \eqref{eq:system}, then $u(t)=||\phi(t,x)-x||\equiv 0$ and the bound
\eqref{eq:containment} is trivial; otherwise, since $u(0)=0$, we must have a time interval  $(0,\delta)$ where $u(t)>0$. Extend this interval maximally (but no further than $\tau$), defining: 
\[
t_1=\max\{\delta \in (0, \tau] \mid u(t)>0 \text{ in } (0,\delta)\}.
\]
We have $\tau \geq t_1>0$, and  
$$ D^+u(t) \leq  Lu(t) + F_S, \quad \forall t\in(0,t_1).$$
We are now in a position to apply a generalization of the Grönwall inequality 
from \cite{davydov2022non} Lemma 11 to yield:
$$u(t) \leq  
\frac{F_S}{L}(e^{Lt}-1)\quad\text{ or }\quad u(t) \leq 
F_S t , \;\;\forall t\in(0,t_1)\,,$$ respectively if $L\neq 0$ or $L=0$, so 
\begin{equation}\label{eq:t1 bound}
  u(t) \leq F_S \,h(t;L) \leq F_S h(\tau;L) \quad \text{ for } 0< t < t_1,  
\end{equation}
where we used the fact that $h(\cdot,L)$ is increasing and $t_1\leq \tau$. 
If $t_1=\tau$, then $F_S h(\tau;L)$ is a bound on the $u(t)$ across the entire interval $(0,\tau]$. If, instead, $t_1< \tau$, then necessarily $u(t_1)=0$ and the system has a periodic orbit of period $t_1$; \eqref{eq:t1 bound} still gives the bound $F_S \,h(\tau;L)$ over the entire period; \red{ since $\phi(t,x)=\phi(t\bmod t_1,x)$, the bound holds for all time}. Thus, \eqref{eq:containment} follows. 
\end{proof}

The Containment Lemma (Lemma \ref{lem:containment}), which provides containment guarantees for a finite time, can be combined with the recurrence property of Definition \ref{defn:T recurrent} and Lemma \ref{lem:infinite-seq} to provide trajectory bounds for all positive times. 

\begin{corollary}[Boundedness of Trajectories]\label{cor:bounded-trajectories}
    Let $S$ be a compact $\tau$-recurrent set. Then it follows that for any $x\in S$,
        \[
            \mathrm{d}(\phi(t,x),S) \leq F_S h(\tau; L),\qquad \forall t\geq 0,
        \]
    where $L :=L_{\mathrm{cl}\mathcal{R}^{\tau}(S)}$. 
    Moreover, the $\tau$-reachable tube $\mathcal{R}^{\tau}(S)$ is invariant.
\end{corollary}

\begin{proof}
Applying Lemma \ref{lem:infinite-seq}, given $x\in S$, we can build a sequence of times $t_n \to \infty$ such that $x_n=\phi(t_n,x) \in S$, and $t_{n+1} - t_n \leq \tau$. Also trajectories starting in $S$ are forward complete. Applying Lemma \ref{lem:containment} starting from $x_n\in S$ we conclude that:
\[
 \mathrm{d}(\phi(t,x), S) \leq F_S h(\tau; L) \text{ for }t\in [t_{n},t_{n+1}]. 
\]
Further, since $t_n$ is arbitrary and $t_n\rightarrow \infty$, we conclude that the bound holds for all time $t\ge0$. 

For the second claim, observe first that if $x\in S$, then $\phi(t,x) \in \mathcal{R}^{\tau}(S)$ for all $t\geq 0$. This follows by placing $t$ in an interval 
$t \in [t_n, t_{n+1})$, and noting that $\phi(t,x) = \phi(t-t_n,x_n)$ with $x_n\in S$ and $t-t_n \leq \tau$. 
Now, if $y\in \mathcal{R}^{\tau}(S)$, then $y=\phi(t',x)$ for $t'\in[0,\tau]$ and $x\in S$, and therefore 
$\phi(t,y) = \phi(t+t',x)\in \mathcal{R}^{\tau}(S)$ $\forall t \ge0$, as required for (forward) invariance. 
\end{proof}

We finalize this section, noting that  Corollary \ref{cor:bounded-trajectories} imbues compact $\tau$-recurrent sets with the same functional property of compact invariant sets, i.e., bounding trajectories. This provides the cornerstone to the development of a recurrence-based stability theory.
\section{Recurrent Lyapunov Functions}\label{ct-rlf}

\dem{Having established the ability to bound trajectories using compact $\tau$-recurrent sets, we now introduce the modified conditions on a function $V:D\to\mathbb{R}_{\geq0}$, that relax the standard Lyapunov conditions for stability. In contrast to the classical counterpart, we do not require $V$ to be monotonically non-increasing along trajectories. Rather, for any given initial $x\in D$, we allow $\tau>0$ units of time to elapse before requiring the function to meet any requirements on its value. This leads to the proposed definition of Recurrent Lyapunov Functions.}
\aem{We now introduce conditions on a function $V:D\to\mathbb{R}_{\geq 0}$ that relax the classical Lyapunov requirement of monotonic non-increase along trajectories: we allow $\tau>0$ units of time to elapse before requiring any condition on $V$, leading to the notion of Recurrent Lyapunov Functions.}


\begin{defn} [Recurrent Lyapunov Function (RLF)]\label{defn:RLF}
Given an equilibrium point $x^*\in D$ of \eqref{eq:system}, a set $S\subseteq D$ satisfying $x^*\in \mathrm{int}(S)$, and $\tau>0$, a \emph{continuous} function $V : D \to \mathbb{R}_{\geq 0}$ is a \textbf{Recurrent Lyapunov Function} over $S$ if:
\begin{enumerate}[$(i)$]
    \item $V$ is \textbf{positive definite} around $x^*$, that is,
            \begin{align}\label{eq:positive-definite}
                V(x)>0,\; \forall x\in D\backslash\{x^*\}, \text{ and } V(x^*)=0.
            \end{align}
    \item $V$ is \textbf{$\tau$-recurrent} over $S$, that is,
            \begin{align}\label{eq:V-tau-recurrent}
                \min_{s \in T_{S}(x;\,\tau)} V(\phi(s,x)) \leq V(x),\quad\forall x\in S.
            \end{align}
\end{enumerate}
We further say $V$ is a \textbf{Strict RLF} (SRLF) over $S$ if, in addition, the inequality in (ii) is strict for $x \in S\setminus\{x^*\}$:
\begin{align}\label{eq:V-strictly-tau-decreasing}
    \min_{s \in T_{S}(x;\,\tau)} V(\phi(s,x)) < V(x),\quad\forall x\in S\backslash\{x^*\}.
\end{align}
\end{defn}

\dem{We make the following remarks about Definition \ref{defn:RLF}. First, the minimum in \eqref{eq:V-tau-recurrent} is taken over the non-necessarily closed set $T_{S}(x;\tau)$. For the $\min$ to be finite as required, there must exist $t\in(0,\tau]$ with $\phi(t,x)\in S$. Second, we only require $V$ to be continuous; while classical Lyapunov theory can be developed for non-differentiable functions, it usually requires increased complexity in the analysis. Our results can be readily stated for only continuous $V$. Finally, the $\tau$-recurrent property \eqref{eq:V-tau-recurrent} acts as a substitute to the standard differential inequality: $\dot V = \nabla V(x)^Tf(x)\leq0$. As we will see next, this condition allows us to substitute the standard invariance property with the more relaxed notion of recurrence.}
\aem{The minimum in~\eqref{eq:V-tau-recurrent} is taken over the (possibly non-closed) set $T_S(x;\tau)$, so finiteness requires the existence of $t \in (0,\tau]$ with $\phi(t,x) \in S$. We require $V$ only to be continuous, avoiding the complications of non-differentiable Lyapunov theory. The $\tau$-recurrent property thus substitutes the standard differential inequality $\dot V = \nabla V(x)^T f(x) \leq 0$, replacing strict invariance with the more relaxed notion of recurrence.}

\begin{lemma}
\label{lem:recurrence of V}
Given any $c\geq0$ and a compact set $S\subseteq D$. If $V:D \rightarrow \mathbb{R}_{\geq 0}$ is continuous and $\tau$-recurrent over $S$ (cf. $(ii)$ in Definition \ref{defn:RLF}), then, the following holds:
{\begin{enumerate}
    \item[(i)] The set $S$ is $\tau$-recurrent. 
    \item[(ii)] The set $V_{\leq c}\cap S = \{x \in S \mid V(x) \leq c\}$ is $\tau$-recurrent.
\end{enumerate}}
\end{lemma}
\ifthenelse{\boolean{lemmas}}{
\begin{proof}
    We start by noting that since $V$ is continuous, it has a finite maximum over  compact $S$; hence, there exists $c$ large enough such that $S\cap V_{\leq c} = S$. As a result, property $(i)$ follows directly from $(ii)$.
    
    To prove $(ii)$ we use the characterization $(ii)$ of Lemma \ref{lem:infinite-seq}. By hypothesis, for any $x\in S\cap V_{\leq c}$, one can find $t'\in (0,\tau]$ such that $\phi(t',x)\in S\cap V_{\leq c}$. Since $S\cap V_{\leq c}$ is compact, by Lemma \ref{lem:infinite-seq} it is $\tau$-recurrent.        
\end{proof}}
{
\noindent \textit{Proof.}~The proof is omitted due to page limits. \hfill \qed
}

We are now ready to present the main result of this section, which states that the existence of an RLF is sufficient to guarantee the stability of the associated equilibrium point.

\begin{theorem}[Stability]\label{thm:recurrence-stability}
Consider system \eqref{eq:system} under Assumption \ref{as:Lipschitz}, with an equilibrium point $x^*\in D$, and a compact
set $S\subseteq D$ satisfying $x^*\in \mathrm{int}(S)$. Then, if $V: D\rightarrow \mathbb{R}_{\geq0}$ is an RLF over $S$, the equilibrium $x^*$ is stable.
\end{theorem}
\begin{proof}
It suffices to show that for any 
$\varepsilon>0$, there exists a set $I\subset B_\varepsilon(x^*)$, $I$ invariant under \eqref{eq:system} and with $x^*\in \mathrm{int}(I)$. Without loss of generality, assume $B_\varepsilon(x^*)\subset S$.

Note first that from Lemma \ref{lem:recurrence of V}, $S$ is $\tau$-recurrent, so trajectories starting in $S$ are forward complete; furthermore, by Lemma \ref{lem:containment}, $L = L_{\mathrm{cl}\left(\mathcal{R}^{\tau}(S)\right)}$ is finite. 

Find $\varepsilon'>0$ small enough such that 
\begin{equation}\label{eq:epsln}
\varepsilon' + F_{\varepsilon'}h(\tau; L)\leq\varepsilon.
\end{equation}
Now let $\alpha=\min_{\varepsilon'\leq \|x-x^*\|\leq \varepsilon} V(x)$; by construction, $\alpha>0$. Select $\beta$ such that $0 < \beta < \alpha$ and introduce the compact set
\begin{equation}\label{eq.omegabeta}
    \Omega_\beta  := \{ x\in B_{\varepsilon'}(x^*) : V(x)\leq \beta\}.
\end{equation}
Claim 1: Let $I:=\mathcal{R}^\tau(\Omega_\beta)$. Then $x^* \in \mathrm{int}(I)$, 
$I \subset B_\varepsilon(x^*)$.

Given \eqref{eq:positive-definite}, $x^* \in \mathrm{int}(\Omega_\beta)$; also $\Omega_\beta \subset I$, so $x^* \in \mathrm{int}(I)$. To establish $I \subset B_\varepsilon(x^*)$, apply the Containment Lemma \ref{lem:containment} to 
$\Omega_\beta$: if $x\in \Omega_\beta$, then for every $t\in (0,\tau]$ we have:
\begin{equation}\label{eq:contain omega beta}
d(\phi(t,x),\Omega_\beta) \leq F_{\varepsilon'} h(\tau;L).
\end{equation}
Since $\Omega_\beta \subset B_{\varepsilon'}(x^*)$, the triangle inequality gives the norm bound $\|\phi(t,x) -x^* \|\leq \varepsilon' + F_{\varepsilon'}h(\tau; L)\leq\varepsilon$, as claimed.  

\noindent 
Claim 2: $\Omega_\beta$ is $\tau$-recurrent. 

Given $x\in \Omega_\beta$, by hypothesis 
there exists $t'\in (0,\tau]$ such that $x'=\phi(t',x) \in S$ and $V(x')\leq \beta$. 
Since $x'\in \mathcal{R}^\tau(\Omega_\beta)$ we have $\|x'-x^*\| \leq \varepsilon$ by Claim 1. However, $V(x')< \alpha$ so $\|x'-x^*\|$ cannot be in the interval $[\varepsilon',\varepsilon]$; so $x' \in B_{\varepsilon'}(x^*)\cap V_{\leq \beta} = \Omega_\beta$, and thus this set is $\tau$-recurrent. 

Now apply Corollary \ref{cor:bounded-trajectories} to conclude that 
$I =\mathcal{R}^\tau(\Omega_\beta)$ is an \emph{invariant} set. So $I$ satisfies the requirements set up at the beginning of the proof. 
\end{proof}

\section{Asymptotic Stability}\label{ct-asymptotic}

\dem{Now that we have proven stability with RLFs, we wish to expand the theory to incorporate asymptotic stability. Similarly to Lyapunov's Direct Method, the extension essentially consists of strengthening the condition from a non-strict inequality to a strict one.}
\aem{We now extend the theory to asymptotic stability via the Strict RLF condition of Definition~\ref{defn:RLF}.}


Since a Strict RLF is also an RLF, all properties established in the previous section continue to hold. In particular, if $S$ is compact, then $S$ is $\tau$-recurrent, trajectories initiating in $S$ are forward complete, and remain bounded for all time by Corollary~\ref{cor:bounded-trajectories}. Moreover, Theorem~\ref{thm:recurrence-stability} guarantees stability of the equilibrium.
The strict inequality in \eqref{eq:V-strictly-tau-decreasing} plays the role of the classical Lyapunov condition $\dot V(x)<0$ for all $x\in S\setminus\{x^*\}$. While classical Lyapunov theory enforces monotonic decrease of $V(\phi(t,x))$ for all $t\ge0$, our condition only requires a strict decrease at some return time within $(0,\tau]$. In particular, it implies that trajectories starting from compact sub-level sets of $V$ return to their interior within $\tau$ units of time, thereby providing the mechanism needed to establish asymptotic stability.

\begin{theorem}[Asymptotic Stability]\label{thm:asymptotic-recurrence-stability}
Consider system \eqref{eq:system} under Assumption \ref{as:Lipschitz}, with an equilibrium point $x^*\in D$ and a compact set $S\subseteq D$ satisfying $x^*\in \mathrm{int}(S)$. Then, if $V: D\rightarrow \mathbb{R}_{\geq0}$ is an SRLF over $S$, the equilibrium $x^*$ is asymptotically stable on the set $S$.
\end{theorem}
\begin{proof}
The stability requirement is already established by Theorem \ref{thm:recurrence-stability} and the fact that an SRLF is also an RLF. 
Thus, we are only left to show the attractivity of $x^*$ on the set $S$.

Consider $x_0\in S\backslash\{x^*\}$, and the trajectory 
$\phi(t,x_0)$, $t\geq 0$ initiating in this point. It is bounded by Corollary \ref{cor:bounded-trajectories}, so consider its $\omega$-limit set $\Omega$, which is a compact set in ${\mathbb{R}^d}$. 

\noindent Claim: $x^* \in \Omega$. 


Assume $x^* \not \in \Omega$. Consider the compact set $\Omega\cap S$, 
\[
\bar{v} := \min_{\Omega \cap S} V(x),
 \ \  \text{and an arbitrary minimizer } \ \ \bar{x}\in \Omega\cap S,
\]
such that $V(\bar x)=\bar v$. Since $V(x)$ is only zero at $x^*\not \in  \Omega\cap S$, we must have $\bar{v}=V(\bar{x})>0$. Now by the SRLF hypothesis at $\bar{x}$, there exists $s\in (0,\tau]$ such that $\phi(s,\bar{x}) \in S$ and
\[
V(\phi(s,\bar{x}))<\bar{v}.
\]
But $\phi(s,\bar{x})\in \Omega$ since the $\omega$-limit is an invariant set, therefore $\phi(s,\bar{x})\in \Omega \cap S$, which contradicts the definition of $\bar{v}$. This establishes the claim.

Now, since $x^*$ is a stable equilibrium point and $x^*\in \Omega$, we must have $\phi(t,x_0)\to x^*$ as $t\to \infty$. \\ 
Explicitly: given $\varepsilon>0$, by stability choose $\delta>0$ such that
$\phi(t,B_\delta(x^*))\subset B_\varepsilon(x^*)$; now from the $\omega$-limit take $t_1:\phi(t_1,x_0) \in B_\delta(x^*)$. Then for any $t\geq t_1$, $\phi(t,x_0)=\phi(t-t_1,\phi(t_1,x_0)) \in B_\varepsilon(x^*)$.
\end{proof}

We point out that the requirement of $S$ to be compact in Theorem \ref{thm:asymptotic-recurrence-stability} can be extended to a global setting, as follows.

\begin{corollary}[Global Asymptotic Stability]\label{cor:global asymptotic stability}
Let Assumption \ref{as:Lipschitz} hold. Consider an equilibrium point $x^*\in D$ of \eqref{eq:system}. Then, if $V: D\rightarrow \mathbb{R}_{\geq0}$ is an SRLF over $D$, and has \textbf{compact sub-level sets} $V_{\leq c}\subset D$, $\forall c\geq0$, then the equilibrium $x^*$ is globally asymptotically stable.
\end{corollary}
\begin{proof}
    Pick any $x\in D\backslash\{x^*\}$ and let $c:=V(x)>0$.   
    Since $V$ is an SRLF over $D$, property \eqref{eq:V-strictly-tau-decreasing} implies there exists $t'\in(0,\tau]$ satisfying 
    \[
    V(\phi(t',x)) < V(x)=c.
    \]
Note then that $\phi(t',x)\in V_{\leq c}$; therefore $V$ satisfies the conditions of an SRLF over the set $S:=V_{\leq c}$. Since $V_{\leq c}$ is compact we can apply Theorem \ref{thm:asymptotic-recurrence-stability} to claim stability of $x^*\in \mathrm{int} V_{\leq c}$ and attractivity of $\phi(t,x)$. Finally, since $x$ was chosen arbitrarily within $D$, the result follows.
\end{proof}

\section{Exponential Stability}\label{ct-exponential}



\dem{In the section we pursue conditions on the function $V$ that enforce \emph{exponential} stability. In classical Lyapunov analysis, this is verified by ensuring an exponential decrease in the Lyapunov function along trajectories:}
\aem{We now seek conditions on $V$ that enforce \emph{exponential} stability. Classical Lyapunov analysis requires an exponential decrease along trajectories,}
\begin{equation}\label{eq:standard exp lyapunov}
V(\phi(t,x)) \leq e^{-\alpha t} V(x),\quad \forall t\geq 0,
\end{equation}
\dem{for some positive constant $\alpha$. As before, such a condition tightly couples the geometry of $V$ to trajectories, significantly complicating the search for such functions.}
\aem{which tightly couples the geometry of $V$ to trajectories and complicates the search.} 

Consistently with the general approach of our paper, we will give conditions on $V$ that ensure exponential convergence at \emph{recurrent} times; from there we will prove exponential convergence for all times. For this purpose, we introduce a definition that relaxes the requirement in \eqref{eq:standard exp lyapunov}. 


\begin{defn}[Exponential Recurrent Lyapunov Function (ERLF)]\label{defn:ERLF}
Given an equilibrium point $x^*\in D$ of \eqref{eq:system} and a set $S\subseteq D$ satisfying $x^*\in \mathrm{int}(S)$. We say that a \emph{continuous} function $V : D \to \mathbb{R}_{\geq 0}$ is an \textbf{Exponential Recurrent Lyapunov Function} over the set $S$ if the following properties hold:
\begin{enumerate}[$(i)$]
    \item $V$ is \textbf{positive definite \red{on $D$} and linearly contained} around $x^*$, that is, there exist constants $a_1,a_2>0$ such that
            \begin{align}\label{eq:linearly-contained}
                a_1\|x - x^*\|\leq V(x)\leq a_2\|x - x^*\|,\;\forall x\in \red{D}.
            \end{align}

    \item $V$ is \textbf{$\alpha$-exponentially $\tau$-recurrent} over $S$, that is, there exist constants $\alpha>0$ and $\tau >0$ such that: 
            \begin{align}\label{eq:exp-tau-recurrent}
                \min_{s \in T_{S}(x;\,\tau)} e^{\alpha s}V(\phi(s,x)) \leq V(x), \quad \forall x\in S\backslash\{x^*\}.\quad \mbox{}
            \end{align}
\end{enumerate}
\end{defn}



Note that an $\alpha$-exponentially $\tau$-recurrent function is always strictly $\tau$-recurrent, but not the other way around. 
Further, in the above definition, we favor a \emph{linear containment} condition in \eqref{eq:linearly-contained}, instead of the standard condition based on class $\mathcal{K}$ functions for two reasons. Firstly, it leads to slightly simpler derivations. Secondly, as we will show in Section \ref{ct-weak converse}, under mild conditions standard norms $\|\cdot \|$ are ERLFs, thus trivially satisfying \eqref{eq:linearly-contained}.

We now show how to use Definition \ref{defn:ERLF} to provide exponentially decreasing bounds for $\|\phi(t,x)-x^*\|$. 

\begin{theorem}[Exponential Stability]\label{thm:tau-exponential-stability}
Consider an equilibrium point $x^*\in D$ of \eqref{eq:system}, and a compact set $S\subseteq D$ satisfying $x^*\in \mathrm{int}(S)$. Suppose Assumption \ref{as:Lipschitz} holds, and let $V: D \to \mathbb{R}_{\geq 0}$ be an \textbf{Exponential Recurrent Lyapunov Function} over the set $S$.
Then, the equilibrium $x^*$ is exponentially stable with rate $\alpha$ on the set $S$. 
That is, for every $x\in S$ we have 
\begin{equation}\label{eq:V-exponentially-stable}
    \|\phi(t,x)-x^*\| \leq C\,e^{-\alpha t}\|x-x^*\|, \quad t\geq 0,
\end{equation}
with {$C:=\frac{a_2}{a_1}e^{\alpha \tau}(1+\bar Lh(\tau;L))$}, 
$L:=L_{\mathrm{cl}\,\mathcal{R}^{\tau}(S)}$, and $\bar L:=\bar L_{\mathrm{cl}\,\mathcal{R}^{\tau}(S)}$.
\end{theorem}

\begin{proof}
Pick any $x\in S$, it suffices to focus on $x\in S\backslash\{x^*\}$. Since ERLF$\Rightarrow$SRLF$\Rightarrow$RLF, it follows as before that
$\phi(t,x)$ is bounded and forward complete. Next, we find a recurrent sequence in this trajectory with 
exponentially decreasing values of $V$, by a similar construction to that in Lemma \ref{lem:infinite-seq}.

\noindent Claim I: there exists a  sequence $\{t_n\}_{n\in\mathbb{N}}$, with $t_0=0$, $\lim_{n\rightarrow \infty} t_n=\infty$ and $t_{n+1}-t_n\in(0,\tau]$  $\forall n$, 
such that $x_n := \phi(t_n,x)\in S\backslash\{x^*\}$ and:
\begin{equation}\label{eq:th2:induction}
e^{\alpha t_{n+1}}V(x_{n+1}) \leq  e^{\alpha t_{n}}V(x_n)  \leq V(x), \, \,\forall \ n\geq 1.
\end{equation}

The sequence is defined by induction. For the base case with $t_0=0$, $x_0=x$, define:
   \begin{align} \label{eq:delta t1 exp}
          t_{1} = \max\{{\arg\min}_{s\in T_{S}(x_0;\tau)}
          e^{\alpha s}V(\phi(s,x_0))\};
     \end{align}
by the hypothesis  \eqref{eq:exp-tau-recurrent} the minimum above exists and is no larger than $V(x)$, and $t_1 \in (0,\tau]$; if there are multiple minimizing times we have selected the largest. Note also that by construction $x_1:=\phi(t_1,x_0) \in S$, and also 
$\phi(t_1,x_0)\neq x^*$ since the latter is an equilibrium and $x_0\neq x^*$. This allows us to repeat the construction inductively:  
   \begin{align} \label{eq:delta tn exp}
          t_{n+1} - t_n = \max\{{\arg\min}_{s\in T_{S}(x_n;\tau)}
          e^{\alpha s}V(\phi(s,x_n))\}.  \quad \mbox{}
     \end{align}
 Invoking again  \eqref{eq:exp-tau-recurrent}, $t_{n+1}-t_n$ is well defined in $(0,\tau]$ with 
\begin{equation}\label{eq:th2:induct}
e^{\alpha (t_{n+1}-t_n)}   V(\phi(t_{n+1}-t_n,x_n)) \leq  V(x_n);
\end{equation}
also note $x_{n+1} = \phi(t_{n+1},x) = \phi(t_{n+1}-t_n,x_n) \in S\backslash\{x^*\}$, and inequalities in \eqref{eq:th2:induction} follow from \eqref{eq:th2:induct}.

It only remains to show that $t_n \to \infty$. 

Assume instead that $t_n\uparrow \bar{t}<\infty$. By continuity of $\phi(\cdot,x)$ 
and compactness of $S$, $x_n=\phi(t_n,x)\rightarrow\phi(\bar{t},x)=:\bar{x}\in S\backslash\{x^*\}$. By the monotonicity in \eqref{eq:th2:induction} we have 
\begin{equation}\label{eq:th2:limit}
e^{\alpha \bar{t}}V(\bar{x})\leq  e^{\alpha t_{n+1}}V(x_{n+1}) \quad \forall n\geq 0.
\end{equation}
We deduce from here that:
\begin{equation}\label{eq:th2:incr}
e^{\alpha (\bar{t}-t_n)}V(\bar{x})\leq  e^{\alpha (t_{n+1}-t_n)}V(x_{n+1}).
\end{equation}
Now pick $n$ large enough so that $\bar{s}:=\bar{t}-t_n \leq \tau$. Note $\bar{s}> 
t_{n+1}-\hl{t_n}$ because $\bar{t}>t_{n+1}$.

Since $\phi(\bar{s},x_n)=\phi(\bar{t},x)=\bar{x}\in S$, the point $\bar{s}$ is in the domain of the minimization in \eqref{eq:delta tn exp}, and gives a result that is no greater than the minimum due to \eqref{eq:th2:incr}. This contradicts the fact that $\bar{s}>t_{n+1}-t_n$, 
as the latter was the \emph{largest} minimizing time index in \eqref{eq:delta tn exp}.
This establishes Claim I.

We will now use \eqref{eq:linearly-contained} and \eqref{eq:th2:induction} to bound the distance from the sequence $x_n$ to equilibrium:
\[
\|x_n-x^*\| \leq \frac{V(x_n)}{a_1} \leq \frac{e^{-\alpha t_n}}{a_1} V(x)=:r_n. 
\]
Let $B_n:=B_{r_n}(x^*)\cap S$. Applying Lemma \ref{lem:containment} on the compact set $B_n\subset S$\red{, with the global constant $L$ in place of $L_{\mathcal{R}^{\tau}(B_n)}$ (allowed since $L_{\mathcal{R}^{\tau}(B_n)}\leq L$ and $h$ is increasing in $L$),} it follows that 
 \begin{equation}\label{eq:bnd exp rn}
\|\phi(t,x)-x^*\| \leq r_n +F_{r_n}h(\tau; L), \; \forall t \in (t_n,t_{n+1}].     
 \end{equation}
{Furthermore, since by Assumption \ref{as:Lipschitz}, $f$ is $\bar L$-Lipschitz on $B_{n}\subset S$, and $f(x^*)=0$ we have $F_{r_n} \leq \bar L r_n$, leading to }
 \begin{align}
\|\phi(t,x)-x^*\| &\leq r_n (1 + {\bar L}h(\tau; L))\\
&= \frac{e^{-\alpha t_n}}{a_1} (1 + \bar Lh(\tau; L)) V(x)     
 \end{align}
for all $t\in(t_n,t_{n+1}]$. For such $t$ we have $t\leq t_n + \tau$, therefore $-t_n \leq \tau - t$ so $e^{-\alpha t_n} \leq e^{\alpha \tau} e^{-\alpha t}$, leading to 
\[
\|\phi(t,x)-x^*\| \leq e^{\alpha \tau} \frac{e^{-\alpha t}}{a_1} (1 +\bar L h(\tau; L)) V(x).
\]
Moreover, since the last bound is independent of $n$, and $n$ was chosen arbitrarily, it must hold 
for all $t\geq0$.
Finally, applying the upper bound $V(x) \leq a_2 \|x-x^*\|$ we establish \eqref{eq:V-exponentially-stable} with $C$ as in the theorem statement \red{(note $C\geq1$, since $a_2\geq a_1$)}. 
\end{proof}

The above theorem demonstrates the exponential stability of an equilibrium point $x^*$, for initial conditions on a compact set $S$, and involving constants  $L$ and $\bar L$ which depend on the set $S$. We now turn to \emph{global} exponential stability results over the entire domain, which require somewhat stronger assumptions.

\begin{corollary}[Global Exponential Stability]\label{cor:global-exponential-stability}
Consider the system \eqref{eq:system}, and assume the vector field $f$ is \textbf{globally} Lipschitz over $D$ with Lipschitz constant $\bar L$. Let $x^*\in D$ be an equilibrium point. Let $V: D \to \mathbb{R}_{\geq 0}$ be an Exponential Recurrent Lyapunov Function over $D$, with constants $\alpha, \tau$. Assume that $V$ has \textbf{compact sublevel sets}.

Then the equilibrium point $x^*$ is globally exponentially stable. In particular, for all $x \in D$ and $t \geq 0$,
\begin{equation}\label{eq:global-exp-bound}
    \|\phi(t,x) - x^*\| \leq C\, e^{-\alpha t} \|x - x^*\|,
\end{equation}
with $C := \frac{a_2}{a_1} e^{(\bar L + \alpha)\red{\tau}}$.
\end{corollary}

\begin{proof}
Let $x \in D$ be arbitrary, and define the compact sublevel set $S := V_{\leq V(x)} \subset D$. Since $V$ is an Exponential Recurrent Lyapunov Function over $D$, it satisfies the ERLF conditions over $S$ as well.

Applying Theorem~\ref{thm:tau-exponential-stability} over $S$ we have:
\[
\|\phi(t,x) - x^*\| \leq C_S e^{-\alpha t} \|x - x^*\| \quad \forall t\geq 0,
\]
where $C_S := \frac{a_2}{a_1} e^{\alpha \tau}(1 + \bar L_S h(\tau; L_S))$.

By assumption, since \( f \) is globally Lipschitz on \( D \) with constant \( \bar L \), both the standard and one-sided Lipschitz constants over \( \mathcal{R}^{\tau}(S) \subseteq D \) satisfy \( L_S \leq \bar L \) and \( \bar L_S \leq \bar L \).

Moreover, the function \( h(\tau; L) := \frac{e^{L\tau} - 1}{L} \) is increasing in \( L \), so we conclude that
\[
C_S \leq \frac{a_2}{a_1} e^{\alpha \tau}(1 + \bar L h(\tau; \bar L)) =\red{\frac{a_2}{a_1} e^{(\alpha+\bar L)\tau}} = C.
\]
Since this bound is independent of the particular choice of $x \in D$, we obtain:
\[
\|\phi(t,x) - x^*\| \leq C e^{-\alpha t} \|x - x^*\|, \quad \forall x \in D, \; \forall t \geq 0.
\]
\end{proof}

\section{Stability Analysis with Norms as RLFs}
\label{ct-weak converse}

The previous sections establish the sufficiency of conditions based on recurrent Lyapunov functions for stability analysis. It is natural to 
inquire about converse results, i.e. the \emph{necessity} of such conditions. 
A first remark is that since our Lyapunov conditions are weaker than the standard ones, necessity follows trivially from the classical converse Lyapunov theory (see, e.g.,~\cite{massera1949liapounoff,sontag_containment}): under asymptotic or exponential stability, a standard $V(x)$ exists (in particular, decreasing along trajectories) that will also satisfy our weaker recurrence conditions. 

Hence we pose a different question, aligned  with our objective of decoupling the Lyapunov choice from the system geometry: can we 
verify recurrence with a generic Lyapunov function? In particular, in Section \ref{ssec:converse theorems} we will use an arbitrary \emph{norm} of the deviation from equilibrium as Lyapunov candidate, and show that (slightly weaker) recurrence conditions follow from the appropriate stability notions. In Section \ref{ssec:time domain implications} we will interpret these conditions as guaranteeing practical notions of stability, which will form the basis of the data-driven verification methods to be proposed in Section~\ref{ct-verify}, and the algorithms developed in Section~\ref{ct-numerical}.


\subsection{Norm-Based Weak Converse Theorems}\label{ssec:converse theorems}

We begin by showing that any norm satisfies the SRLF condition on compact subsets of the domain of attraction of an \emph{asymptotically stable equilibrium}, provided a neighborhood of the equilibrium itself is excluded. To that end we provide the following relaxation of SRLFs.

\begin{defn}[$\varepsilon$-Strict Recurrent Lyapunov Func. ($\varepsilon$-SRLF)]
\label{defn:epsilon-strict RLF}
Given an equilibrium $x^*\in D$ of \eqref{eq:system}, a set $S\subseteq D$ satisfying $x^*\in \mathrm{int}(S)$, and $\tau>0$. We say that  
a continuous function $V:D\to\mathbb{R}_{\geq0}$ is an  
\textbf{$\varepsilon$-Strict Recurrent Lyapunov Function} over $S$ 
if:

\begin{enumerate}[$(i)$]
    \item $V$ is \textbf{positive definite} around $x^*$, that is,
            \begin{align}\label{eq:positive-definite-2}
                V(x)>0,\; \forall x\in D\backslash\{x^*\}, \text{ and } V(x^*)=0.
            \end{align} 
    \item $V$ is $\varepsilon$-\textbf{strictly $\tau$-recurrent} over $S$, i.e., there exists $\varepsilon>0$, with $B_\varepsilon(x^*)\subset \mathrm{int}(S)$, such that
    \begin{equation}\label{eq:eps-SRLF}
    \min_{s\in T_S(x;\tau)} V(\phi(s,x)) < V(x),
    \qquad \forall x\in K_\varepsilon,
    \end{equation}
    where $K_\varepsilon:=\mathrm{cl}\left(S\setminus B_\varepsilon(x^*)\right)$.
\end{enumerate}
\end{defn}


We now show that norms can indeed be considered as natural $\varepsilon$-SRLFs.

\begin{theorem}[Asymptotic Stability Implies Norm is $\varepsilon$-SRLF]\label{thm:norm-srlf}
Given system \eqref{eq:system}. Let \( x^* \in D \) be an asymptotically stable equilibrium on a compact set \( S \subseteq D \) (Definition \ref{defn:asymptotic stability}) satisfying $x^*\in \mathrm{int}(S)$. Let \( \|\cdot\| \) be any norm on \( {\mathbb{R}^d} \). Then, for any $\varepsilon>0$ such that $B_\varepsilon(x^*)\subset \mathrm{int}(S)$, there exists a finite $\tau>0$ such that the function \( V(x) := \|x - x^*\| \) is a $\varepsilon$-Strict Recurrent Lyapunov Function ($\varepsilon$-SRLF) over \( S \).
\end{theorem}
\begin{proof}
Let $V(x)= \|x-x^*\|$ and consider a fixed $\varepsilon>0$ such that $B_\varepsilon(x^*)\subset \mathrm{int}(S)$. 
By definition $V$ satisfies \eqref{eq:positive-definite-2}. Also introduce $\delta > 0$ (exists due to equilibrium stability) such that for any $x_0$ with $V(x_0)<\delta$, $V(\phi(t,x_0)) < \varepsilon$ for all $t\geq 0$.  

Claim I: there exists a finite $\tau>0$ such that for all $x\in S$, and $t\geq \tau$, $V(\phi(t,x)) < \varepsilon$. We prove this claim:

For each $x \in S$, due to asymptotic stability we can find a time $t(x)$ such that $V(\phi(t(x),x)) = \|\phi(t(x),x) - x^*\|<\delta$. 
By continuity of the flow with respect to initial conditions we can further find $r(x)>0$ such that for any $x' \in \mathrm{int} B_{r(x)}(x)$, $V(\phi(t(x),x')) <\delta$. Given the choice of $\delta$, we have: 
\begin{equation}\label{eq.settling}
    V(\phi(s,x')) < \varepsilon \quad \forall s \geq t(x), \ \  x'\in   \mathrm{int} B_{r(x)}(x).
\end{equation}
Now, the family $\{\mathrm{int}  B_{r(x)}(x)\}_{x\in S}$ constitutes an open cover of the compact set $S$: select a finite subcover $\{ \mathrm{int}  B_{r_k}(x_k)\}_{k=1}^K$. Define $\tau = \max_{k} t(x_k)$. We conclude from \eqref{eq.settling} that $V(\phi(s,x')) < \varepsilon$ for all $x'\in S$, $s \geq \tau$, as claimed. 

To verify condition \eqref{eq:eps-SRLF} of our $\varepsilon$-SRLF definition, consider now an initial condition 
$x\in K_\varepsilon:=\mathrm{cl}\left(S\setminus B_\varepsilon(x^*)\right)$. We have
\begin{equation}
    V(x)=\|x-x^*\| \geq \varepsilon > V(\phi(\tau,x)) \geq \min_{s\in T_S(x;\tau)} V(\phi(s,x));
\end{equation}
note for the last step that $\phi(\tau,x) \in S$ since $B_\varepsilon(x^*) \subset S$. 
Thus, $V(x)=\|x-x^*\|$ is $\varepsilon$-strictly $\tau$-recurrent.
\end{proof}

Having established that asymptotic stability implies $\varepsilon$-strict $\tau$-recurrence of standard norms, we now turn to exponential stability. In contrast to the asymptotic case—where strict recurrence can only be guaranteed away from a neighborhood of the equilibrium—exponential stability allows us to retain the full domain $S$ while relaxing the certified rate of convergence. 

Nevertheless, for consistency with the verification framework to be developed in Section~\ref{ct-verify}, we also formulate a converse result that accommodates the possibility of $\varepsilon$-ERLFs.

\begin{defn}[$\varepsilon$-Exponential Recurrent Lyapunov Function ($\varepsilon$-ERLF)]
\label{defn:epsilon-erlf}
Given an equilibrium $x^*\in D$ of \eqref{eq:system}, a set $S\subseteq D$ satisfying $x^*\in \mathrm{int}(S)$, and constants $\alpha>0$ and $\tau>0$, 
a continuous function $V:D\to\mathbb{R}_{\geq 0}$ is said to be an 
\textbf{$\varepsilon$-Exponential Recurrent Lyapunov Function} over $S$ if:

\begin{enumerate}[(i)]
    \item $V$ is \textbf{linearly contained} around $x^*$ on $S$, i.e.,
    \begin{equation}
        a_1\|x-x^*\| \le V(x) \le a_2\|x-x^*\|,
        \qquad \forall x\in S,
    \end{equation}
    for some constants $a_1,a_2>0$;

    \item $V$ is $\varepsilon$-\textbf{strictly $\alpha$-exponentially $\tau$-recurrent} over $S$, i.e., there exists $\varepsilon>0$ with $B_\varepsilon(x^*)\subset \mathrm{int}(S)$ such that
    \begin{equation}\label{eq:eps-ERLF}
        \min_{s\in T_S(x;\tau)} e^{\alpha s} V(\phi(s,x))
        \le V(x),
        \qquad \forall x\in K_\varepsilon, 
    \end{equation}
    where $K_\varepsilon:= \mathrm{cl}\left(S\setminus B_\varepsilon(x^*)\right)$.
\end{enumerate}
\end{defn}


\begin{theorem}[Exponential Stability Implies Norm is $\varepsilon$-ERLF]
\label{thm:norm-erlf}
Consider system~\eqref{eq:system}, and let \( x^* \in D \) be a \(\lambda\)-exponentially stable equilibrium on a compact set $S$ satisfying $x^*\in \mathrm{int}(S)$ (Definition~\ref{defn:exp stability}). 
Then for any \(0<\alpha<\lambda\) and any $\varepsilon\ge 0$ such that $B_\varepsilon(x^*)\subset \mathrm{int}(S)$, the function \(V(x) := \|x - x^*\|\) is an $\varepsilon$-Exponential Recurrent Lyapunov Function over \(S\) for any $\tau$ satisfying
\begin{equation}\label{eq:exp-tau}
\tau\geq  \frac{1}{\lambda-\alpha}\ln\left(C\frac{R}{r}\right),
\end{equation}
where $C,\lambda$ are given in Definition \ref{defn:exp stability}, and $r,R$ are positive constants satisfying $B_{r}(x^*)\subseteq S \subseteq B_{R}(x^*)$.
\end{theorem}
\begin{remark}
As mentioned before, in Theorem~\ref{thm:norm-erlf}, the parameter $\varepsilon$ may be taken equal to zero: exponential stability implies that $V(x)=\|x-x^*\|$ satisfies the full ERLF condition over $S$. We formulate the result for $\varepsilon$-ERLFs to maintain consistency with the asymptotic case and, more importantly, with the verification framework to be developed in Section~\ref{ct-verify}, where practical considerations may require the use of $\varepsilon>0$.
\end{remark}
\begin{proof}
W.l.o.g. we prove the theorem  statement for $\varepsilon=0$. Let \( V(x) := \|x - x^*\| \), where \(\|\cdot\|\) is the norm satisfying the exponential stability Definition \ref{defn:exp stability}. By hypothesis, 
for all \( x \in S \) and all \( t \geq 0 \), we have $\|\phi(t,x)-x^*\|\leq C e^{-\lambda t}\|x - x^*\|.$

Applying the above condition at $t=\tau$ and noting that $C e^{-\lambda \tau} \leq \frac{r}{R} e^{-\alpha \tau}$ from \eqref{eq:exp-tau}, we have:
\begin{align}\label{eq:exp-bound-tau}
\|\phi(\tau,x)-x^*\|
&\le \frac{r}{R} e^{-\alpha \tau}\|x-x^*\|.
\end{align}
Recalling $x\in B_{R}(x^*)$ we conclude in particular that $\phi(\tau,x)\in B_{r}(x^*)\subseteq S$. Moreover, from \eqref{eq:exp-bound-tau} we have
\begin{align}
e^{\alpha \tau}\|\phi(\tau,x)-x^*\|
&\le \frac{r}{R}\|x-x^*\|
\le \|x-x^*\|.
\end{align}
Hence, for $V(x)=\|x-x^*\|$ we have:
\begin{align*}
\min_{t\in T_S(x;\tau)} e^{\alpha t} V(\phi(t,x)) &\le e^{\alpha \tau}V(\phi(\tau,x)) \le V(x),
\end{align*}
satisfying the ERLF property over $S$, and thus, the $\varepsilon$-ERLF property for all $\varepsilon\ge0$ s.t. $B_\varepsilon(x^*)\subset\mathrm{int}(S)$.
\end{proof}

\subsection{Stability Implications of Weak Converse Conditions}
\label{ssec:time domain implications}

The $\varepsilon$-strict recurrence conditions of the previous subsection are weaker versions of the respective RLF conditions. A natural question is to identify the implications of these weakened conditions in regard to stability. We will now show that they imply adequate notions of \emph{practical} stability. 

\begin{defn}[Practical Stability]\label{defn:practical-stability}
An equilibrium $x^*$ is \textbf{practically stable} with precision $\varepsilon'>0$ if there exists $\delta>0$ s.t.
\[
\|x-x^*\|\le \delta
\ \Longrightarrow\
\|\phi(t,x)-x^*\|\le \varepsilon',
\quad \forall t\ge 0.
\]
\end{defn}

\begin{defn}[Practical Asymptotic Stability on $S$]\label{defn:practical-as-stability}
The equilibrium $x^*$ is \emph{practically asymptotically stable}
on $S$ with precision $\varepsilon'>0$, 
if
\begin{enumerate}[(i)]
    \item $x^*$ is practically stable with precision $\varepsilon'$;
    \item $\displaystyle 
    \limsup_{t\to\infty}
    \|\phi(t,x)-x^*\|
    \le \varepsilon'
    \quad \forall x\in S.
    $
\end{enumerate}
\end{defn}


\begin{defn}[Practical Exponential Stability on $S$]\label{defn:practical-exp-stability}
Let $S\subset D$ and $x^*\in\mathrm{int}(S)$. The equilibrium $x^*$ is 
\emph{practically exponentially stable on $S$ with rate $\lambda>0$ and precision $\varepsilon'$} 
if there exists $C\ge 1$ such that for all $x\in S$ and all $t\ge 0$,
\begin{equation} \label{eq:def-pract-exp}
\|\phi(t,x)-x^*\| \le \max\{\varepsilon',C e^{-\lambda t}\|x-x^*\|\}.
\end{equation}
\end{defn}


The following lemma, a direct consequence of Lemmas \ref{lem:infinite-seq} and  \ref{lem:containment}, provides preliminary practical stability implications of $\varepsilon$-SRLFs and $\varepsilon$-ERLFs based on norms. 

\begin{lemma}[Practical Stability from {norm} $\varepsilon$-RLFs] 
\label{lem:practical-stability-eps}
Let $x^*$ be an equilibrium of \eqref{eq:system}, satisfying Assumption~\ref{as:Lipschitz}; 
let $S\subset D$ be compact, with $B_\varepsilon(x^*)\subset \mathrm{int}(S)$ for {$\varepsilon > 0$}. If {$V(x)=\|x-x^*\|$} is either an $\varepsilon$-SRLF or an $\varepsilon$-ERLF over $S$ with parameter $\tau$,
then $B_\varepsilon(x^*)$ is $\tau$-recurrent, and  $x^*$ is practically stable with precision 
\begin{equation}\label{eq:varepsilon'}
\varepsilon' := \varepsilon + F_\varepsilon\, h(\tau;L),
\end{equation}
where $F_\varepsilon := \sup_{x\in B_\varepsilon(x^*)}\|f(x)\|$ and 
$L:=L_{\mathrm{cl}\,\mathcal{R}^\tau(B_\varepsilon(x^*))}$  
as in Lemma~\ref{lem:containment}.
\end{lemma}
\begin{proof}
We first show that $B_\varepsilon(x^*)$ is $\tau$-recurrent, by applying Lemma \ref{lem:infinite-seq}, condition (ii); namely, given $x\in B_\varepsilon(x^*)$, there exists $t\in (0,\tau]$ such that
$\phi(t,x) \in B_\varepsilon(x^*)$. 
Clearly, it suffices to check this condition for  $x\in\partial B_\varepsilon(x^*)$, i.e. with $V(x)=\|x-x^*\|=\varepsilon$. Note that such points belong to the set $K_\varepsilon$ for both cases of $\varepsilon$-SRLF \eqref{eq:eps-SRLF} and $\varepsilon$-ERLF \eqref{eq:eps-ERLF}. Using these conditions we conclude in both cases that there exist $t\in (0,\tau]$ such that $V(\phi(t,x))<\varepsilon$, as required by Lemma \ref{lem:infinite-seq}; thus
$B_\varepsilon(x^*)$ is $\tau$-recurrent. 

We can now apply Corollary \ref{cor:bounded-trajectories} over $B_\varepsilon(x^*)$; for any $x$ satisfying $\|x-x^*\|\leq \varepsilon$, 
 $d(\phi(t,x), B_\varepsilon(x^*)) \leq F_\varepsilon h(\tau;L)$ for all $t\ge 0$, 
with the given definitions of $F_\varepsilon$, $h$ and $L$. Therefore: 
\[
\|\phi(t,x)-x^*\|\leq \varepsilon +  F_\varepsilon h(\tau;L)=\varepsilon' \ \ \forall t\ge 0;
\]
this establishes the claimed practical stability precision. 
\end{proof}


We are now ready to characterize the time domain implications of the $\varepsilon$-SRLF and $\varepsilon$-ERLF conditions. 


\begin{theorem}[Practical Asymptotic Stability for $\varepsilon$-SRLF]
\label{thm:practical AS-eps}
Let $x^*$ be an equilibrium of \eqref{eq:system}, satisfying Assumption~\ref{as:Lipschitz}; let $S\subset D$ be compact, with $B_\varepsilon(x^*)\subset \mathrm{int}(S)$ for {$\varepsilon > 0$}. If {$V(x)=\|x-x^*\|$} is an $\varepsilon$-SRLF with parameter $\tau$, then $x^*$ is practically asymptotically stable on $S$ with precision $\varepsilon'$ given by \eqref{eq:varepsilon'}.


\end{theorem}

\begin{proof} 
Observe first that the hypothesis implies that the compact set $S$ is $\tau$-recurrent, using Lemma \ref{lem:infinite-seq}. Indeed, for an initial condition $x\in S$: if $x\in K_\varepsilon$, \eqref{eq:eps-SRLF} implies there exists $s\in (0,\tau]$ such that $\phi(s,x)\in S$, as required in condition (ii) of Lemma \ref{lem:infinite-seq}; if, instead, $x\in \mathrm{int}B_\varepsilon(x^*)$, this condition is immediate. 

Hence, by Corollary~\ref{cor:bounded-trajectories}, every trajectory $\phi(t,x)$ with $x\in S$ is forward complete and bounded, and thus has a non-empty, compact, and invariant $\omega$-limit set.

Moreover, Lemma~\ref{lem:practical-stability-eps} implies that $x^*$ is practically stable with precision $\varepsilon'$ as in \eqref{eq:varepsilon'}. In particular, if $x\in B_\varepsilon(x^*)$, then $\|\phi(t,x)-x^*\|\le \varepsilon'$ for all $t\ge0$, condition (i) in Definition \ref{defn:practical-as-stability}. To establish condition (ii), it suffices to show that any trajectory initiated in $x\in K_\varepsilon:=\mathrm{cl}(S\setminus B_\varepsilon(x^*))$ reaches the ball $B_\varepsilon(x^*)$ at some finite time. 


Suppose by contradiction that, for  $x\in K_\varepsilon$, $\phi(t,x)\notin \mathrm{int}(B_\varepsilon(x^*))$ for all $t\ge0$. 
Let $\Omega$ be its $\omega$-limit set. Then $\Omega\cap S$ is \red{nonempty (since $S$ is compact and visited at times $t_n\to\infty$),} compact\red{,} and contained in $K_\varepsilon$. Let $\bar x\in\Omega\cap S$ minimize $V$ on $\Omega \cap S$. The $\varepsilon$-SRLF property gives $s\in(0,\tau]$ with $V(\phi(s,\bar x))<V(\bar x)$ and $\phi(s,\bar x)\in S$. By invariance of $\Omega$, $\phi(s,\bar x)\in\Omega$ as well, contradicting minimality. 

We have thus shown that every trajectory from $K_\varepsilon$ enters $B_\varepsilon(x^*)$ in finite time, and practical stability yields $
\limsup_{t\to\infty}\|\phi(t,x)-x^*\|\le \varepsilon'.$
\end{proof}

\begin{theorem}[Practical Exponential Stability for $\varepsilon$-ERLF]
\label{thm:practical ES-eps}
Let $x^*$ be an equilibrium of \eqref{eq:system}, satisfying Assumption~\ref{as:Lipschitz}; let $S\subset D$ be compact, with $B_\varepsilon(x^*)\subset \mathrm{int}(S)$ for {$\varepsilon > 0$}. If {$V(x)=\|x-x^*\|$} is an $\varepsilon$-ERLF with parameters $(\alpha,\tau)$, then $x^*$ is practically exponentially stable on $S$ with rate $\alpha$ and precision $\varepsilon'$ given by \eqref{eq:varepsilon'}.
\end{theorem}
\begin{proof} 
For initial conditions $x\in B_\varepsilon(x^*)$ the bound \eqref{eq:def-pract-exp}
holds trivially from the practical stability of Lemma \ref{lem:practical-stability-eps}. 

We thus focus on initial conditions $x\in K_\varepsilon$, and carry out an iterative construction as in Theorem~\ref{thm:tau-exponential-stability}, with some adjustments. 
Let $t_0=0$, $x_0=x$, and define as in \eqref{eq:delta t1 exp}:
 \begin{align} \label{eq:delta t1 pract exp}
          t_{1} = \max\{{\arg\min}_{s\in T_{S}(x;\tau)}
          e^{\alpha s}V(\phi(s,x))\} \in (0,\tau]
     \end{align}
and $x_1:=\phi(t_1,x)\in S$. By hypothesis \eqref{eq:eps-ERLF} we have $e^{\alpha t_{1}}V(x_1)  \leq V(x)$. 
If $\|x_1 - x^*\| < \varepsilon$, then define $N:=1$ and stop the construction.
If, instead, $x_1 \in K_\varepsilon$, we take another step defining $(t_2, x_2)$, and so on. Specifically, while the sequence $x_n$ remains in $K_\varepsilon$, we define
as in \eqref{eq:delta tn exp}:
  \begin{align} \label{eq:pract delta tn exp}
          t_{n+1} - t_n = \max\{{\arg\min}_{s\in T_{S}(x_n;\tau)}
          e^{\alpha s}V(\phi(s,x_n))\}, 
     \end{align}
and $x_{n+1} = \phi(t_{n+1} -t_n,x_n) = \phi(t_{n+1},x)$. By hypothesis we have the recursive bound 
\begin{equation}
    e^{\alpha t_{n+1}}V(x_{n+1})  \leq e^{\alpha t_{n}}V(x_n) \leq V(x),
\end{equation}
exactly as in \eqref{eq:th2:induction}. While $x_n$ remains in $K_\varepsilon$, we deduce that 
\begin{equation}\label{eq:exp bound pract n}
\|x_n - x^*\|\leq e^{-\alpha t_n } \|x- x^*\|.    
\end{equation}
Claim I: this recursion must stop after a finite number of steps, reaching $N:=n+1$ 
such that  $\|x_N - x^*\| < \varepsilon$. 

Suppose, instead, that $x_n \in K_\varepsilon$ for all $n\geq 1$. From \eqref{eq:exp bound pract n} we would have 
\[
0 <  \varepsilon  \leq \|x_n-x^*\| \leq  e^{-\alpha t_{n}}  \|x- x^*\|   \  \ \forall n, 
\]
which implies $\{t_n\}$ is a \emph{bounded}, increasing sequence. Define $\bar{t} = \lim t_n$ and 
$\bar{x} = \phi(\bar{t},x)$. Then $\bar{x} = \lim {x_n} \in K_\varepsilon$ due to compactness,  
and we are in an identical situation as in the proof of Theorem~\ref{thm:tau-exponential-stability} (around equation \eqref{eq:th2:incr}), where we showed a contradiction stemming from the finite limit $\bar{t}$. The same argument applies, reaching the same contradiction,  establishing our Claim. As a consequence, we conclude from  practical stability that 
\begin{equation}\label{eq:exp bound from tN}    
\|\phi(t,x)-x^*\|\le \varepsilon'  \mbox{ for all }t\ge t_N.
\end{equation}
Claim II: for any $t\in (0,t_N]$, we have the bound
\begin{equation}\label{eq:exp bound pract}
\|\phi(t,x) - x^*\| \leq C e^{-\alpha t}  \|x - x^*\|,
\end{equation}
where $C:=e^{\alpha \tau} (1 +\bar L h(\tau; L))$, with constants defined as in Theorem~\ref{thm:tau-exponential-stability}\ \footnote{With the simplification that $a_1=a_2=1$ in this case.}, and similar proof which is now sketched. 

Select $n\in \{0,1,\ldots N-1\}$ such that 
$t\in (t_n,t_{n+1}]$, and denote $r_n:=\|x_n-x^*\|$. Apply Lemma \ref{lem:containment} over $S\cap B_{r_n}(x^*)$ to obtain the bound \eqref{eq:bnd exp rn}: 
$\|\phi(t,x) - \red{x^*}\|\leq r_n + F_{r_n}h(\tau;L)$. 
From the Lipschitz condition we have $F_{r_n}\leq \bar{L}r_n$, which gives
\[
\|\phi(t,x) - x^*\| \le r_n(1 +\bar L h(\tau; L))=r_n C e^{-\alpha \tau}.
\]
Now $t - t_n \leq \tau$, so  $e^{-\alpha \tau}\leq e^{-\alpha(t- t_{n})}$
which implies
\[
\|\phi(t,x) - x^*\| \le C e^{-\alpha t} e^{\alpha t_n} r_n 
\leq C e^{-\alpha t} \|x-x^*\|,
\]
where the last inequality follows from
\eqref{eq:exp bound pract n}. This establishes Claim II. 

We have two upper bounds, \eqref{eq:exp bound pract} for $t\leq t_N$ and \eqref{eq:exp bound from tN} for $t\ge t_N$; the max of both bounds holds for all $t \geq 0$.
\end{proof}

\section{Verification of $\varepsilon$-Exponential RLFs}\label{ct-verify}

\dem{So far, we have introduced Recurrent Lyapunov Functions and established guarantees for stability, asymptotic stability, and exponential stability under appropriate recurrence conditions (RLFs, SRLFs, and ERLFs, respectively). We further showed that, under mild conditions, norms satisfy $\varepsilon$-versions of the SRLF and ERLF conditions, yielding practical asymptotic and practical exponential stability guarantees.}

We now leverage the universality of norms as RLFs to develop a practical mechanism for verifying $\varepsilon$-Exponential Recurrent Lyapunov Functions directly from trajectory data. {The essence of the procedure is to give sufficient conditions to guarantee that the  $\varepsilon$-ERLF condition \eqref{eq:eps-ERLF} holds over a certain ball of initial conditions, and then cover the domain of interest with an adequate number of such balls.}

Our approach is related to ideas from topological entropy, which characterize dynamical systems through coverings of trajectory segments and the growth rate of distinguishable trajectories~\cite{Adler1965,Bowen1971}, as well as their extensions to control through notions such as topological feedback entropy and invariance entropy~\cite{Nair2004,Colonius2009}.

\subsection{Trajectory-based Verification of ERLF Condition}\label{ssec:verification-neighborhood}

We start by deriving conditions to verify the ERLF recurrence condition locally around a neighborhood of a certain initial condition $x$. We will make use of the \emph{signed distance} function  from \( x \in {\mathbb{R}^d} \) to \( S\subset{\mathbb{R}^d} \), given by:
\[
\mathrm{sd}(x, S) := 
\begin{cases}
\mathrm{d}(x, \partial S), & \text{if } x \notin S, \\
- \mathrm{d}(x, \partial S), & \text{if } x \in S.
\end{cases}
\]
It is not difficult to show that sd is non-expansive, i.e. $|\mathrm{sd}(x,S)-\mathrm{sd}(y,S)|\leq \|x-y\|$.

\begin{theorem}[Trajectory-based Verification of ERLF Property]\label{thm:verification-norm-ERLF}
Consider the system~\eqref{eq:system}, an equilibrium point~\(x^* \in D\), a compact set \(S \subseteq D\), and a given initial condition $x\in S$. Given constants \( r>0 \), \(\alpha > 0\), and \(\tau > 0\), let \( L = L_{\mathcal{R}^{\tau}(S\cup B_r(x))} \), and assume there exists \( t \in (0,\tau] \) satisfying simultaneously:
\begin{subequations}\label{eq:verification-condition}
\begin{align}
e^{\alpha t}\bigl(\|\phi(t,x)-x^*\|+r e^{Lt}\bigr)&\leq \|x - x^*\|-r,\label{eq:verify-exp-decreasing}\\[2pt]
\mathrm{sd}(\phi(t,x),S)+r e^{Lt}&\leq 0.\label{eq:verify-feasibility}
\end{align}
\end{subequations}
Then, the function \(V(y)=\|y - x^*\|\) satisfies condition \eqref{eq:eps-ERLF} on \(B_r(x)\), i.e.,
\[
\min_{s\in T_S(y;\tau)} e^{\alpha s}V(\phi(s,y)) \le V(y),
\qquad \forall y\in B_r(x).
\]
\end{theorem}
\begin{proof}
Let \(V(\cdot):=\|\cdot-x^*\|\). Since system~\eqref{eq:system} is \(L\)-one-sided-Lipschitz on \(\mathcal R^\tau(S\cup B_r(x))\), for every \(y\in B_r(x)\) and every \(t\in[0,\tau]\),
\begin{equation}\label{eq:Gronwall-bound}
\|\phi(t,x)-\phi(t,y)\|\le e^{Lt}\|x-y\|\le r e^{Lt}.
\end{equation}
Hence,
\begin{equation}\label{eq:phi-upper-new}
\|\phi(t,y)-x^*\|
\le
\|\phi(t,x)-x^*\|+r e^{Lt}.
\end{equation}
Let \(t^*\in(0,\tau]\) satisfy \eqref{eq:verify-exp-decreasing}--\eqref{eq:verify-feasibility}. Then, for any \(y\in B_r(x)\),
\begin{align*}
e^{\alpha t^*}\|\phi(t^*,y)-x^*\|
&\le e^{\alpha t^*}\bigl(\|\phi(t^*,x)-x^*\|+r e^{Lt^*}\bigr) \\
&\le \|x-x^*\|-r
\le \|y-x^*\|,
\end{align*}
where the last inequality follows from \(y\in B_r(x)\). Therefore,
\[
e^{\alpha t^*}V(\phi(t^*,y))\le V(y), \qquad \forall y\in B_r(x).
\]
It remains to show that \(t^*\in T_S(y;\tau)\). To this end, note that:
\begin{align*}
\mathrm{sd}(\phi(t^*,y),S)
&\le \mathrm{sd}(\phi(t^*,x),S)+\|\phi(t^*,y)-\phi(t^*,x)\|\\
&\le \mathrm{sd}(\phi(t^*,x),S)+r e^{Lt^*}
\le 0,
\end{align*}
where the first step uses the non-expansiveness of sd, the second follows from the Lipschitz bound \eqref{eq:Gronwall-bound}, and the last from condition \eqref{eq:verify-feasibility}.
Hence \(\phi(t^*,y)\in S\), $t^*\in(0,\tau]$, and therefore \(t^*\in T_S(y;\tau)\). Consequently,
\[
\min_{s\in T_S(y;\tau)} e^{\alpha s}V(\phi(s,y)) \le V(y),
\qquad \forall y\in B_r(x),
\]
which proves the claim.
\end{proof}
\begin{remark}\label{rem:verify feasibility}
    We make a few remarks regarding Theorem \ref{thm:verification-norm-ERLF}. Checking \eqref{eq:verification-condition} requires the simulation of one trajectory over a finite time. It is critical that \eqref{eq:verify-exp-decreasing} and \eqref{eq:verify-feasibility} both are satisfied by the same $t$. However, when the set $S=B_R(x^*)$, \eqref{eq:verify-exp-decreasing} implies that $\phi(t^*,x)$ is automatically in $S$, avoiding the need for     \eqref{eq:verify-feasibility}, thus simplifying the verification process.
\end{remark}

\subsection{Verifying $\varepsilon$-ERLFs over a Set and its Complexity}\label{ssec:sample-complexity}

\dem{Theorem~\ref{thm:verification-norm-ERLF} provides a local mechanism for verifying
that a norm Lyapunov function satisfies the ERLF recurrence condition \eqref{eq:eps-ERLF} over a neighborhood \(B_r(x)\) of a certain $x\in S$. To verify $\varepsilon$-ERLF over $S$ as in Definition \ref{defn:epsilon-erlf}, we must extend this verification to  $K_\varepsilon=\mathrm{cl} (S\backslash B_\varepsilon(x^*))$. This suggests covering $K_\varepsilon$ by an appropriate set of balls, each verified as in Theorem~\ref{thm:verification-norm-ERLF} .
Consider first the possibility of using a fixed radius $r>0$ for all such verification balls. Since there must be a ball (say with center $\tilde{x}$) that touches $\partial B_\varepsilon(x^*)$, we must have $\|\tilde{x}-x^*\|\leq r+\varepsilon$. This implies the right-hand side of \eqref{eq:verify-exp-decreasing} is bounded by $\varepsilon$, which in turn requires $r<\varepsilon$. Since the balls must cover a set of possibly large radius $R := \sup_{x\in S}\|x-x^*\|$, this requires a number of balls (and thus a number of  trajectory evaluations) of at least order $O((R/\varepsilon)^d)$, leading to a prohibitive sample complexity. The natural alternative is to adapt the radius $r$ of the covering balls to the distance from equilibrium. We now describe one such procedure, for simplicity restricting attention to \(S=B_R(x^*)\). We construct a layered covering of the \(K_\varepsilon=\mathrm{cl}(S\setminus B_\varepsilon(x^*))\) using concentric annuli centered at \(x^*\). Specifically,
\[
A_i := \{x\in S : R_i \ge \|x-x^*\| > R_{i+1}\}, \qquad i=0,\dots,n-1,
\]
where \(R_0:=R\), \(R_i:=\rho^i R\), and \(\rho\in(0,1)\). Each layer \(A_i\) will be covered by balls of radius \(r_i:=\mu R_i\), with \(\mu\in(0,1)\). We have thus a geometric progression in both radii, with resolution increasing in the direction of the equilibrium; this will be key to controlling the total number of required balls.}

\aem{Theorem~\ref{thm:verification-norm-ERLF} provides a local mechanism for verifying the ERLF recurrence condition~\eqref{eq:eps-ERLF} on a neighborhood $B_r(x)$ of $x\in S$. To verify $\varepsilon$-ERLF over $S$ (Definition~\ref{defn:epsilon-erlf}), we must cover $K_\varepsilon=\mathrm{cl}(S\setminus B_\varepsilon(x^*))$ by such neighborhoods.
A uniform radius $r>0$ is infeasible: any ball touching $\partial B_\varepsilon(x^*)$ satisfies $\|\tilde{x}-x^*\|\le r+\varepsilon$, so~\eqref{eq:verify-exp-decreasing} forces $r<\varepsilon$, which in turn requires $\Omega((R/\varepsilon)^d)$ balls to cover $S$ of radius $R:=\sup_{x\in S}\|x-x^*\|$---prohibitive. The natural alternative is to adapt $r$ to the distance from equilibrium; restricting to $S=B_R(x^*)$, we construct a layered covering of $K_\varepsilon$ via concentric annuli
\[
A_i := \{x\in S : R_i \ge \|x-x^*\| > R_{i+1}\}, \qquad i=0,\dots,n-1,
\]
where $R_0:=R$, $R_i:=\rho^i R$, and $\rho\in(0,1)$. Each layer $A_i$ is covered by balls of radius $r_i:=\mu R_i$, with $\mu\in(0,1)$. The geometric progression in both radii---resolution increasing toward the equilibrium---will be key to controlling the total number of required balls.}

Applying Theorem~\ref{thm:verification-norm-ERLF} for the balls of each annulus\red{, restricted to $B_{r_i}(x)\cap S$ so that all trajectories involved remain in $\mathcal{R}^{\tau}(S)$ and the constant $L$ applies}, we must verify, for all \(x\in A_i\), the condition 
\begin{equation}\label{eq:key-ineq}
\min_{t \in T_{S}(x;\tau)} e^{\alpha t}
\bigl( \|\phi(t,x) - x^*\| + r_i e^{L t} \bigr)
\leq \|x - x^*\| - r_i.
\end{equation}
Recall from Remark~\ref{rem:verify feasibility}: when \(S = B_R(x^*)\), \eqref{eq:verify-exp-decreasing} alone suffices for verification.
The key question is how to choose the relevant parameters $(\rho,\mu)$ to allow for the preceding condition to hold for all $i$, if at all possible. Clearly, the answer depends on the underlying dynamics. 

\dem{The following result provides, for an exponentially stable equilibrium, a recipe for parameter choices for which the verification is feasible and with moderate complexity.}



\begin{theorem}[Sample Complexity for Verifying $\varepsilon$-ERLF]
\label{thm:sample-complexity}
Consider system~\eqref{eq:system}, an equilibrium \(x^*\in D\subseteq\mathbb{R}^d\), and a constant \(R>0\). Let \(S := B_R(x^*)\) and $L := L_{\mathcal{R}^{\tau}(S)}$. 
Suppose \(x^*\) is \(\lambda\)-exponentially stable over \(S\) with constant \(C\ge1\), and choose \(\alpha\in(0,\lambda)\) such that \(\beta:=\frac{\lambda-\alpha}{\lambda+L}\in(0,1)\). Define \(C_\beta:=C^{\frac{1-\beta}{\beta}}\).

Then, the $\varepsilon$-ERLF recurrence condition for \(V(x)=\|x-x^*\|\) over
\(K_\varepsilon:=\mathrm{cl}(S\setminus B_\varepsilon(x^*))\) 
can be verified with at most 
\begin{equation}\label{eq:complexity}
N(\varepsilon)
:= \ln\!\left(\frac{R}{\varepsilon}\right)3^{d+1}\frac{de}{\beta}(2+C_\beta)^{\frac{d}{\beta}}
= O\!\left(\log\!\left(\frac{R}{\varepsilon}\right)\right)
\end{equation}
trajectory evaluations of duration
\[
\tau=\frac{1}{\lambda-\alpha}\ln\!\left(C(2+C_\beta)e^{\beta/d}\right).
\]
\end{theorem}
\begin{proof}


Since \(x^*\) is \(\lambda\)-exponentially stable over \(B_R(x^*)\), for all \(x\in A_i\),
\[
e^{\alpha t}\|\phi(t,x)-x^*\|
\le C e^{(\alpha-\lambda)t}\|x-x^*\|
\le C e^{(\alpha-\lambda)\tau} R_i.
\]
Hence the left-hand side of \eqref{eq:key-ineq} is upper bounded by
\[
R_i\bigl(C e^{(\alpha-\lambda)\tau} + \mu e^{(\alpha+L)\tau}\bigr).
\]
Using \(\|x-x^*\|\ge R_{i+1}=\rho R_i\) and \(r_i=\mu R_i\),
\[
\|x-x^*\|-r_i \ge R_i(\rho-\mu).
\]
Thus \eqref{eq:key-ineq} holds if
\begin{equation}\label{eq:mu-rho-condition}
C e^{(\alpha-\lambda)\tau} + \mu e^{(\alpha+L)\tau} \le \rho - \mu.
\end{equation}
Choosing now
$\tau(\mu)=\frac{1}{\lambda-\alpha}\ln\!\left(\frac{C}{\mu^\beta}\right)$, 
so that 
\begin{align*}
&C e^{(\alpha-\lambda)\tau(\mu)}=Ce^{-\ln\left(\frac{C}{\mu^\beta}\right)}=\mu^\beta,\\
&\mu e^{(\alpha+L)\tau}=\mu e^{\frac{\alpha + L}{\lambda-\alpha}\ln\left(\frac{C}{\mu^\beta}\right)} =C_\beta\,\mu^\beta,
\end{align*}
turns \eqref{eq:mu-rho-condition} into
\begin{equation}\label{eq:mu-rho-condition-2}
(1+C_\beta)\mu^\beta \le \rho - \mu \;\;\iff\;\; (1+C_\beta)\mu^\beta +\mu \le {\rho}. 
\end{equation}

We will show next that the following choice satisfies \eqref{eq:mu-rho-condition-2}:
\begin{equation}\label{eq:mu(rho)}
\mu(\rho):=\left(\frac{\rho}{2+C_\beta}\right)^{1/\beta}.
\end{equation}
Note first that since $\mu(\rho), \beta \in(0,1)$ we have $\mu(\rho)<\mu(\rho)^\beta$.
Thus
\[
(1+C_\beta)\mu(\rho)^\beta+\mu(\rho)<(2+C_\beta)\mu(\rho)^\beta= {\rho}.
\]

Next, since \(R_i=\rho^i R\), the number of layers satisfies
\begin{equation}\label{eq:number-annuli}
\rho^n R \le \varepsilon
\;\Rightarrow\;
n=\left \lceil \ln\!\left(\frac{R}{\varepsilon}\right)\frac{1}{\ln(\rho^{-1})}\right\rceil.
\end{equation}

To bound the number of balls per layer, we use standard relations from \cite{wainwright2019high} between covering numbers $\mathcal{N}(\mathcal{K},r)$ (minimum number of $r$-balls $B_r$ covering  $\mathcal{K}$) and packing numbers $\mathcal{P}(\mathcal{K},r)$ (maximum number of disjoint balls of radius $r/2$ whose centers belong to $\mathcal{K}$) of a compact set $\mathcal{K}$, namely:
\[
 \mathcal N(\mathcal{K},r)\le \mathcal P(\mathcal{K},r)\le \frac{\mathrm{vol}(\mathcal{K}+B_{\frac{r}{2}})}{\mathrm{vol}(B_{\frac{r}{2}})}.
\]
Applying these to \(A_i\!=\!B_{R_i}(x^*)\setminus B_{R_{i+1}}(x^*)\!\subset\! B_{R_i}(x^*)\) gives\footnote{In the rest of the proof, the ball center $x^*$ is left implicit.}
\begin{align}
\mathcal N(A_i,r_i)
&\le \mathcal N(B_{R_i},r_i)
\le \mathcal P(B_{R_i},r_i) \nonumber
    \le \frac{\mathrm{vol}(B_{R_i+\frac{r_i}{2}})}{\mathrm{vol}(B_{\frac{r_i}{2}})}\\
&
    \!=\!\frac{\mathrm{vol}(B_{2R_i+r_i})}{\mathrm{vol}(B_{r_i})}
    \!=\!\left(\frac{2+\mu(\rho)}{\mu(\rho)}\right)^d \!\le\! \frac{3^d}{\mu(\rho)^d}, \label{eq:annuli-cover}
\end{align}
where we used $\mathrm{vol}(B_{cr})=c^d\mathrm{vol}(B_{r})$,   $\frac{\mathrm{vol}(B_{R})}{\mathrm{vol}(B_{r})}=\left(\frac{R}{r}\right)^d$, \(r_i=\mu(\rho)R_i\) and $\mu(\rho)\in(0,1)$.

Hence the total number of trajectories is bounded by
\begin{equation}\label{eq:bound-prelim}
n\,\mathcal N(A_i,r_i)
\le
\frac{3^d}{\mu(\rho)^d}\left \lceil\frac{\ln(R/\varepsilon)}{\ln(\rho^{-1})}\right\rceil\le\frac{3^{d+1}}{\mu(\rho)^d}\frac{\ln(R/\varepsilon)}{\ln(\rho^{-1})}.
\end{equation}
The above bound is minimized over \(\rho\in(0,1)\) at \(\rho^*=e^{-\beta/d}\), which substituted back and using \eqref{eq:mu(rho)} gives\red{, for $\varepsilon\le R\,e^{-\beta/(2d)}$,} the bound 
\[
N(\varepsilon)
:=
\ln\!\left(\frac{R}{\varepsilon}\right)3^{d+1}\frac{de}{\beta}(2+C_\beta)^{\frac{d}{\beta}}
\]
on the total number of trajectory evaluations required to certify the $\varepsilon$-ERLF condition for $V(x)=\|x-x^*\|$ over $K_\varepsilon$. 
\end{proof}

%

\hl{
\begin{remark}[Performance vs Complexity Trade-off]
Theorem~\ref{thm:sample-complexity} highlights the intrinsic trade-off between the performance gap $\lambda-\alpha$ and the sample complexity: as $\lambda-\alpha\to0^+$, $\beta\to0^+$ and $C_\beta\to\infty$, while at the opposite extreme if $\lambda\to\infty$, $\beta,C_\beta\to 1$ and the bound simplifies to $N(\varepsilon)=O(d\ln(R/\varepsilon)\,9^d)$ with arbitrarily small $\tau\to 0^+$. A constant gap thus yields sample complexity exponentially better than $O((R/\varepsilon)^d)$.
\end{remark}
 }

\section{Numerical Methods}\label{ct-numerical}

Building on the sample-complexity result of Section~\ref{ssec:sample-complexity}, this section develops parallelizable algorithms that certify the $\varepsilon$-ERLF property of $V(x)=\|x-x^*\|$ directly from trajectory data.
To simplify the exposition, we take $x^*=0$ throughout and write $B_r:=B_r(0)$ for $r>0$; the extension to a general equilibrium is immediate. {Here $\|\cdot\|$ is the working norm, $V(x)=\|x\|$, and $L$ is its one-sided Lipschitz constant~\eqref{eq:mujacobian}; $B_r(x)$ is the ball in this norm. We also write $Q_h(x):=\{y:\|y-x\|_\infty\le h\}$ for the cube ($\infty$-ball) of half-spacing $h$, with $Q_h:=Q_h(0)$, and let $c_\infty$ be the norm-equivalence constant in $\|x\|\le c_\infty\|x\|_\infty$ ($c_\infty=1$ for $\|\cdot\|_\infty$, $c_\infty=\sqrt{d}$ for $\|\cdot\|_2$), so that $Q_h(x)\subseteq B_{c_\infty h}(x)$.}

We present two complementary verification tools, addressing dual questions about the $\varepsilon$-ERLF condition:
\begin{enumerate}
    \item[(T1)] \emph{Best rate over a fixed region.} Given ${Q_R}$, find the largest $\alpha>0$ for which $V$ satisfies the $\varepsilon$-ERLF condition on ${Q_R} \setminus B_\varepsilon$ (Algorithm~\ref{alg:region-verification}, Section~\ref{ssec:verification region}).
    \item[(T2)] \emph{Largest certified region for a fixed rate.} Given ${Q_R}$ and a target rate $\alpha>0$, find a subset $S \subseteq {Q_R}$ such that $V$ satisfies the $\varepsilon$-ERLF condition on $S \setminus B_\varepsilon$ (Algorithm~\ref{alg:grow-alpha-roa}, Section~\ref{ssec:roa mining}).
\end{enumerate}
Both tools rely on a common set of supporting routines, described next in Section~\ref{ssec:supporting methods}. 

\subsection{Supporting Tools}\label{ssec:supporting methods}

Three constructions underlie both verification tools: a layered candidate grid, a per-ball rate-feasibility check, and a refinement step. We assume here that $\tau$ and $L$ are given; their selection is treated later, separately for each method in the respective  subsections~\ref{ssec:verification region} and~\ref{ssec:roa mining}.

\mypara{Initial grid setup}
Given the inner radius $\varepsilon$ and outer radius $R$, we discretize {$Q_R \setminus B_\varepsilon$} by a layered grid of $\mathcal{O}(3^d m)$ candidate {cells}, where $m$ is the number of layers. Layer $\ell\in\{1,\dots,m\}$ contributes $3^d-1$ {cells} (excluding the origin), {each a cube $Q_{h_\ell}(x_i)$ of half-spacing}
\[
    {h_\ell := 3^{\ell-1}\,\varepsilon/c_\infty.}
\]
\textcolor{black}{A cell $Q_h(x)$ is certified through the verification ball $B_{c_\infty h}(x)\supseteq Q_h(x)$. The innermost half-spacing $\varepsilon/c_\infty$ makes the excluded cube $Q_{\varepsilon/c_\infty}\subseteq B_\varepsilon$, so the full $\varepsilon$-ERLF neighborhood is removed.}
Choosing $m$ such that {$R \le 3^m\varepsilon/c_\infty$} ensures that {these cells, and hence the balls $B_{c_\infty h_\ell}(x)$, cover $Q_R \setminus B_\varepsilon$.} This construction is exponentially more efficient than a uniform $\varepsilon$-grid, which would require $\mathcal{O}((R/\varepsilon)^d)$ points and is also aligned with our sample complexity bounds of Theorem \ref{thm:sample-complexity}; an example for $m=2$, $d=2$ is shown in Figure~\ref{fig:initial grid}.


\begin{figure}[!htbp]
    \centering
    \includegraphics[width=0.7\linewidth]{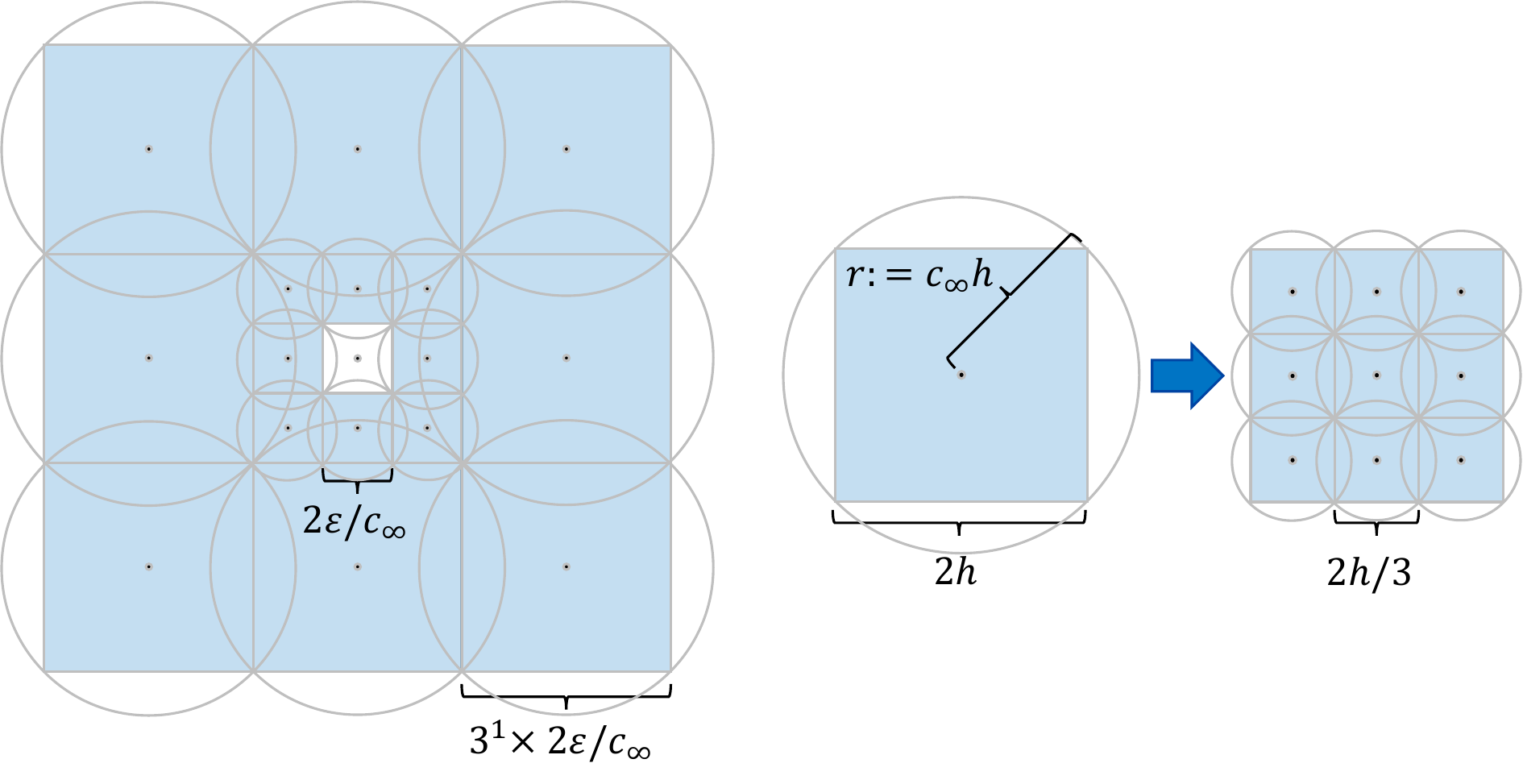}
    \caption{Left: initial layered grid for \textcolor{black}{$R = 3^m\varepsilon/c_\infty$} with $m=2$ layers; \textcolor{black}{the cells are cubes $Q_{h_\ell}(x_i)$ of half-spacing $h_\ell=3^{\ell-1}\varepsilon/c_\infty$, each certified through the containing $\|\cdot\|$-ball $B_{c_\infty h_\ell}(x)$ (here Euclidean, $c_\infty=\sqrt{d}$);} dots are grid points and the central white cell is the excluded cube \textcolor{black}{$Q_{\varepsilon/c_\infty}\subseteq B_\varepsilon$} around $x^*$. Right: the $3^d$-way subdivision performed by \textcolor{black}{$\mathrm{Split}(x,h)$, splitting a cube $Q_h(x)$ into $3^d$ sub-cubes $Q_{h/3}$, each with verification ball $B_{c_\infty h/3}$.}}
    \label{fig:initial grid}
\end{figure}

\mypara{Verifying decay of a ball at a given point}
For a given center $x \in \mathbb{R}^d$ and radius $r>0$, {we let} $\alpha_{\max}(x,r;S)$ {denote} the largest $\alpha$ for which the ball $B_r(x)$ satisfies the verification condition~\eqref{eq:verification-condition} relative to a reference set $S$. \textcolor{black}{This is a one-dimensional maximization of $\alpha$ over $t\in(0,\tau]$ subject to~\eqref{eq:verify-exp-decreasing}--\eqref{eq:verify-feasibility}, by Theorem~\ref{thm:verification-norm-ERLF}.} \textcolor{black}{When $S$ is a sub-level set of $V$, which, for $S=Q_R$, holds exactly when $\|\cdot\|=\|\cdot\|_\infty$,  condition~\eqref{eq:verify-feasibility} is implied by~\eqref{eq:verify-exp-decreasing} (see Remark~\ref{rem:verify feasibility}); we then write $\alpha_{\max}(x,r):=\alpha_{\max}(x,r;Q_R)$.} If no feasible $\alpha$ exists, then $\alpha_{\max}(x,r):=-\infty$.


\mypara{Splitting failed points}
When {$\alpha_{\max}(x,r;S)$} fails to certify a {box of half-spacing $h$} (i.e., $\alpha_{\max}(x,{c_\infty h})<0$), the failure may reflect that the {box} is simply too large for the local Lipschitz bound to suffice. In that case {$\mathrm{Split}(x,h)$ subdivides it} into $3^d$ {sub-boxes of half-spacing $h/3$} centered on a uniform grid {(offsets $\pm\tfrac{2}{3}h$ per axis), illustrated} for $d=2$ in the right panel of Figure~\ref{fig:initial grid}. The split is fully parallelizable and can be applied recursively.



\subsection{Find Best Decay Rate $\alpha_{\min}$ over a Given Region ${Q_R}$}\label{ssec:verification region}

We now address Tool (T1): given an outer radius $R$, find the largest $\alpha>0$ for which $V(x)={\|x\|}$ satisfies the $\varepsilon$-ERLF condition on ${Q_R}$. The approach combines the supporting routines of Section~\ref{ssec:supporting methods} with a procedure for selecting $\tau$ and estimating $L$ tailored to the fixed-set nature of this setting.

\mypara{Selecting $\tau$ and estimating $L$}
Both constants are determined from a single set of trajectories initialized on the boundary of ${Q_R}$. We start by constructing a uniform grid $\Gamma\subset\partial {Q_R}$,  with points separated $2h>0$ in the $\|\cdot\|$ norm,
and simulate each trajectory $\phi(t,x)$, $x\in \Gamma$.

\emph{Selection of $\tau$.} We seek a value of $\tau$ for which every trajectory starting in $\Gamma$ either returns to ${Q_R}$ or enters $B_\varepsilon$ within $[0,\tau]$. Starting from a small candidate value, we increase $\tau$ until this condition is met or until it exceeds a maximum horizon $\tau_{\max}$; in the latter case the procedure is aborted and $R$ can be reduced before retrying.

\emph{Estimation of $L$.} Once $\tau$ is fixed, the worst-case excursion radius
\[
    R_{\max} := \max_{x\in \Gamma,\; t\in[0,\tau]} {\|\phi(t,x)\|_\infty}
\]
provides a data-driven approximation of the radius of the reachable set $\mathcal{R}^{\tau}({Q_R})$. We thus pick $R'>R_{\max}$ (with a small slack for discretization) and bound the one-sided Lipschitz constant by
\[
    L := \sup_{z\in {Q_{R'}}} {\mu}\!\left(\tfrac{\partial f}{\partial z}(z)\right),
\]
{where $\mu$ is the matrix measure of the working norm (cf.~\eqref{eq:mujacobian}).} In the examples below, this supremum is obtained either by {maximizing $\mu$ over the corners of $Q_{R'}$, which is exact since $\partial f/\partial z$ is affine} (Section~\ref{ssec:experiments1}), or in closed form (Section~\ref{ssec:roa mining}); data-driven estimators of $L$ are also possible~\cite{knuth2020planning-fcc}.

To certify robustness to discretization, we verify that an $h$-neighborhood of every $x\in \Gamma$ remains inside ${Q_{R'}}$ over $[0,\tau]$:
\begin{equation}\label{eq:robust-grid}
    \max_{t\in[0,\tau]}\;\max_{x\in \Gamma}\;{\|\phi(t,x)\|_\infty}+h\,e^{tL}\;\le\; R'.
\end{equation}
If \eqref{eq:robust-grid} fails, the grid $\Gamma$ is refined and the procedure repeated.

\mypara{Region verification algorithm}
With $\tau$ and $L$ in hand, the procedure is summarized in Algorithm~\ref{alg:region-verification}. We construct the initial layered grid of Section~\ref{ssec:supporting methods}, producing candidate pairs $G=\{(x_i,{h_i})\}$ covering ${Q_R}\setminus B_\varepsilon$. For each pair we compute two quantities:
\begin{align*}
\underline{\alpha}_i &:= \alpha_{\max}(x_i,{c_\infty\,h_i}),
& &\text{(\emph{lower}: certifies the {box at} $x_i$)}\\
\overline{\alpha}_i &:= \alpha_{\max}(x_i,0),
& &\text{(\emph{upper}: certifies the center $x_i$)}.
\end{align*}
The certified rate over the entire grid is $\underline{\alpha}_{i^*}$ with $i^*:=\arg\min_i \underline{\alpha}_i$. The relative gap at the worst point,
\[
    \delta_{i^*}:=\frac{\overline{\alpha}_{i^*}-\underline{\alpha}_{i^*}}{\underline{\alpha}_{i^*}},
\]
measures how much can still be gained by further refinement. Whenever $\delta_{i^*}$ exceeds a threshold $\theta\in(0,1)$, the $k$ {boxes} with smallest $\underline{\alpha}_i$ are subdivided using $\textsc{Split}(\cdot)$ and the loop is repeated; the procedure stops once $\delta_{i^*}\le\theta$ or after a maximum of $m$ refinements. If successful, it returns a uniform lower bound on $\alpha$ such that $V(x)={\|x\|}$ satisfies the $\varepsilon$-ERLF condition on ${Q_R}\setminus B_\varepsilon$, as guaranteed by Theorem~\ref{thm:verification-norm-ERLF}.

\setcounter{algocf}{0}
\begin{algorithm*}[tb]
\caption{$\mathrm{Find}\text{-}\alpha_{\min}(R, \varepsilon,\theta)$ — Best decay rate $\alpha$ over ${Q_R}\setminus B_\varepsilon$ via parallel ball certification}
\label{alg:region-verification}

\KwIn{Outer radius $R>0$, inner radius $\varepsilon\in(0,R)$, gap threshold $\theta\in(0,1)$, maximum refinements $m$}
\KwOut{Lower bound on the certified value of $\alpha$}

Construct initial grid $G\gets\{(x_i,{h_i})\}$ covering $Q_R\setminus B_\varepsilon$; set $\texttt{counter}\gets 0$\;
For all $(x_i,{h_i})\in G$, compute $\underline{\alpha}_i\gets\alpha_{\max}(x_i,{c_\infty\,h_i})$ and $\overline{\alpha}_i\gets\alpha_{\max}(x_i,0)$\tcp*{lower / upper bounds}
Let $i^*\gets\arg\min_{i\in[|G|]}\underline{\alpha}_i$\tcp*{worst box in the grid}
\While{ $\texttt{counter}\le m-1$ \textbf{and} $(\overline{\alpha}_{i^*}-\underline{\alpha}_{i^*})/\underline{\alpha}_{i^*}>\theta$ }{
    Replace the $k$ {boxes} in $G$ with smallest $\underline{\alpha}_i$ by $\textsc{Split}(\cdot)$; $\texttt{counter}\gets\texttt{counter}+1$\;
    Recompute $\underline{\alpha}_i$ and $\overline{\alpha}_i$ over the updated $G$; update $i^*$\;
}
\Return $\underline{\alpha}_{i^*}$
\end{algorithm*}

\mypara{Numerical example: bilinear systems}\label{ssec:experiments1}
We apply Algorithm~\ref{alg:region-verification} to two bilinear families parametrized by their nonlinearity strength $\sigma$. {We use the Euclidean norm $\|\cdot\|_2$ ($c_\infty=\sqrt{d}$), since the spiraling linearization makes $\mu_\infty$ over-conservative.} In two dimensions,
\begin{align}\label{eq:experiment dynamics}
    \begin{bmatrix}
    \dot{x}_1\\\dot{x}_2
    \end{bmatrix}=
    \begin{bmatrix}
    0& 2\\-1& -1
    \end{bmatrix}
    \begin{bmatrix}
    x_1\\ x_2
    \end{bmatrix}+B_1
    \begin{bmatrix}
    x_1^2\\x_1x_2\\ x_2^2
    \end{bmatrix},
\end{align}
and in three dimensions,
\begin{align}\label{eq:experiment dynamics2}
    \begin{bmatrix}
    \dot{x}_1\\\dot{x}_2\\\dot{x}_3
    \end{bmatrix}=
    \begin{bmatrix}
    -1& 0&0\\0.5& -1&0\\0.5& 0.5&-1
    \end{bmatrix}
    \begin{bmatrix}
    x_1\\ x_2\\ x_3
    \end{bmatrix}+B_2
    \begin{bmatrix}
    x_1^2\\\vdots\\ x_3^2
    \end{bmatrix},
\end{align}
with $B_1\in\mathbb{R}^{2\times3}$ and $B_2\in\mathbb{R}^{3\times 6}$ whose entries are drawn i.i.d.\ from $\mathcal{N}(0,\sigma)$. Increasing $\sigma$ amplifies the nonlinearity.
In all experiments we set $R=0.7$ and $\varepsilon=0.01$, and parallelize trajectory simulation with the TorchODE toolbox~\cite{lienen2022torchode}. {The one-sided Lipschitz constant over $Q_{R'}$ is obtained by maximizing $\mu_2(\partial f/\partial z)$ over its corners, exact since $\partial f/\partial z$ is affine.}
Figure~\ref{fig:exp-2} {overlays the phase portrait of~\eqref{eq:experiment dynamics} ($\sigma=0.3$) with the certified region and the adaptive cube sizes produced by Algorithm~\ref{alg:region-verification} across} $Q_R\setminus B_\varepsilon$.

\begin{figure}[htbp]
    \;\;\includegraphics[width=0.425\textwidth]{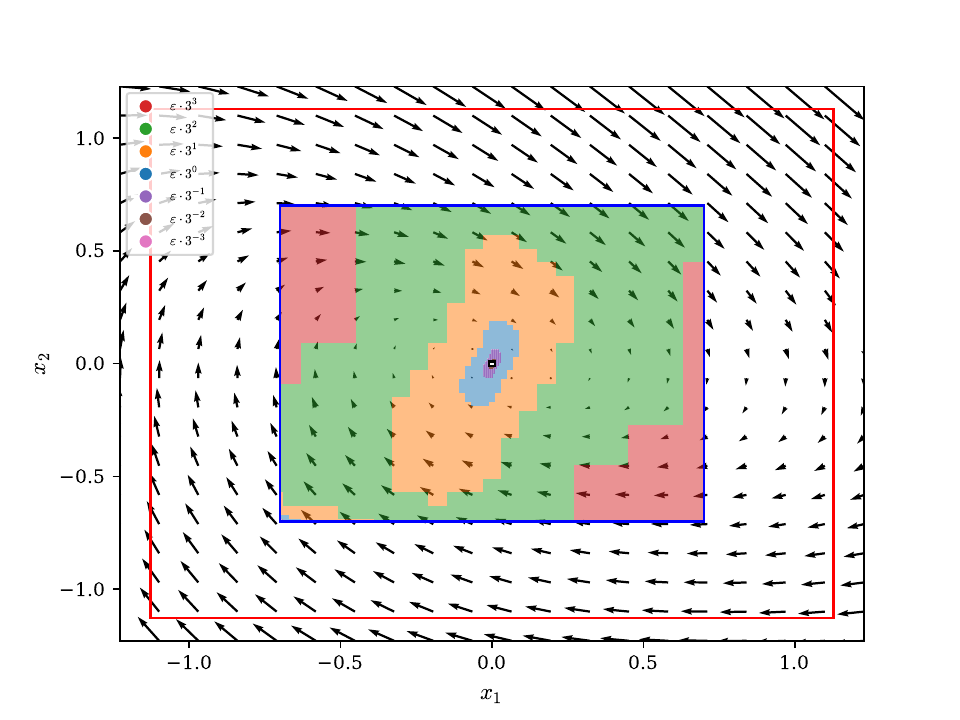}
    \caption{{Certified region for~\eqref{eq:experiment dynamics} ($\sigma=0.3$) over its phase portrait. Blue: region $Q_R$; red: $L$-validity box $Q_{R'}$; black: excluded cube $Q_{\varepsilon/c_\infty}$. Colors give the verification radius $r=\varepsilon\cdot 3^{k}$ ($=c_\infty h$) of the certifying ball (Algorithm~\ref{alg:region-verification}). \hl{$L\simeq 0.58$, $\tau=2.20$}.}}
    \label{fig:exp-2}
\end{figure}

\hl{Table~\ref{table:bilinear} compares the certified rate $\alpha$ and the wall-clock time of Algorithm~\ref{alg:region-verification} against SOSTOOLS~\cite{sostools} on systems~\eqref{eq:experiment dynamics} and~\eqref{eq:experiment dynamics2}. Across both dimensions and all values of $\sigma$, our algorithm certifies a tighter rate $\alpha$.}
\hl{In both dimensions, the runtime of Algorithm~\ref{alg:region-verification} grows with the nonlinearity strength $\sigma$, as larger one-sided Lipschitz constants over $Q_{R'}$ force deeper refinement. While SOSTOOLS is faster in $d=2$, where the underlying SDPs are small, the gap reverses in three dimensions: Algorithm~\ref{alg:region-verification} is faster for every $\sigma$ ($49.5$--$130.8\,$s versus $54.9$--$632.6\,$s), and its advantage widens as $\sigma$ increases, with SOSTOOLS scaling steeply (from $55\,$s at $\sigma=0.1$ to $633\,$s at $\sigma=0.5$).}

\begin{table}[htbp]
\setlength{\tabcolsep}{3pt}
\centering
\begin{tabular}{|c|ccc|ccc|}
\hline
System
   & \multicolumn{3}{c|}{2D system}
   & \multicolumn{3}{c|}{3D system} \\
\hline
$\sigma:$ & 0.3 & 0.6 & 1 & 0.1 & 0.3 & 0.5 \\
\hline\hline
Alg.~\ref{alg:region-verification} $\alpha$:
   & \textbf{\hl{0.422}} & \textbf{\hl{0.442}} & \textbf{\hl{0.336}}
   & \textbf{\hl{0.613}} & \textbf{\hl{0.485}} & \textbf{\hl{0.289}} \\
\hline SoS $\alpha$:
   & 0.360 & 0.247 & 0.223 & 0.309 & 0.341 & 0.213 \\
\hline\hline
Alg.~\ref{alg:region-verification} T (s):
   & \hl{8.3} & \hl{37.9} & \hl{102.6} & \textbf{\hl{49.5}} & \textbf{\hl{88.4}} & \textbf{\hl{130.8}} \\
\hline SoS T (s):
   & \textbf{0.97} & \textbf{1.06} & \textbf{0.94}
   & {54.89} & 276.10 & 632.55 \\
\hline
\end{tabular}
\caption{Certified rate $\alpha$ and runtime: Algorithm~\ref{alg:region-verification} versus SOSTOOLS on the bilinear systems~\eqref{eq:experiment dynamics} and~\eqref{eq:experiment dynamics2}.}
\label{table:bilinear}
\end{table}

\subsection{Estimation of $\alpha$-Regions of Attraction}\label{ssec:roa mining}

We now address Tool (T2): given a target decay rate $\alpha>0$ and a candidate outer radius $R$, identify a subset $S\subseteq Q_R$ over which $V(x)=\|x\|$ satisfies the $\varepsilon$-ERLF condition with rate at least $\alpha$. The structural difference with Subsection~\ref{ssec:verification region} is that the set $S$ is now the \emph{output} of the procedure; consequently, the set-feasibility condition~\eqref{eq:verify-feasibility} is no longer automatic. This shapes both the estimation of $L$ and the structure of the algorithm below. For this method, $\tau$ is chosen a priori, trading off trajectory length against the conservativeness of the output set $S$.

\mypara{Estimating $L$ given $\tau$}
For a candidate subset $S\subseteq Q_R$ to be $\tau$-recurrent, any trajectory $\phi(\cdot,x)$ starting in $S$ must fall into one of two regimes over $[0,\tau]$: either (i)~it remains inside $Q_R$ throughout, or (ii)~it leaves $Q_R$ for some excursion of duration less than $\tau$ and returns. Regime (i) is handled by estimating $L$ over $Q_R$ exactly as in Subsection~\ref{ssec:verification region}. To cover regime (ii), we 
start with a uniform boundary grid $\Gamma\subset\partial Q_R$ (with spacing $2h$), and partition it according to whether each sample's trajectory returns to $Q_R$ within $\tau$:
\[
    \Gamma_{\mathrm{ret}}\!:=\!\{x\!\in\!\Gamma\mid\exists\,t\!\in\!(0,\tau]\!:\phi(t,x)\!\in\! Q_R\}, \quad \Gamma_{\mathrm{nr}}\!:=\!\Gamma\!\setminus\!\Gamma_{\mathrm{ret}}.
\]
The returning samples play the role of $\Gamma$ in Subsection~\ref{ssec:verification region}: setting $R_{\max}:=\max_{x\in\Gamma_{\mathrm{ret}},\,t\in[0,\tau]}{\|\phi(t,x)\|_\infty}$, we choose $R'>R_{\max}$ and bound $L$ over ${Q_{R'}}$ via ${\mu}(\partial f/\partial z)$ as before. The discretization-robustness check~\eqref{eq:robust-grid} carries over verbatim, with $\Gamma_{\mathrm{ret}}$ in place of $\Gamma$.
The complementary samples $\Gamma_{\mathrm{nr}}$ require a one-sided check ensuring that trajectories from an $h$-neighborhood of any non-returning sample also fail to return:
\begin{equation}\label{eq:robust-grid-nr}
    \min_{t\in[0,\tau]}\;\min_{x\in\Gamma_{\mathrm{nr}}}\;\|\phi(t,x)\| - h\,e^{tL}\;>\;R.
\end{equation}
If either check fails, $\Gamma$ is refined and the partition recomputed.

\mypara{Region-growing algorithm}
With $\tau$ given as input and $L$ estimated as above, Algorithm~\ref{alg:grow-alpha-roa} grows the certified set $S$ from the initial layered grid of Section~\ref{ssec:supporting methods}. Each candidate pair $(x_i,{h_i})$ is tested by $\alpha_{\max}(\cdot,\cdot;\cdot)$; {boxes} satisfying the rate threshold are added to a set $\texttt{Positives}$, while failing {boxes} are recursively refined using $\textsc{Split}(\cdot)$ up to depth $m$.

The choice of the reference set inside $\alpha_{\max}(\cdot,\cdot;\cdot)$ controls whether the feasibility condition~\eqref{eq:verify-feasibility} is enforced as the certified region grows. Enforcing it live is computationally expensive: each box would need to be re-tested whenever $\texttt{Positives}$ expands. The Boolean flag $\texttt{Trim}$ decouples the two conditions:
\begin{itemize}
    \item with $\texttt{Trim}=\texttt{False}$, only~\eqref{eq:verify-exp-decreasing} (exponential decay) is checked;
    \item with $\texttt{Trim}=\texttt{True}$,~\eqref{eq:verify-feasibility} is also checked against the candidate region $\bigcup_{(x,{h})\in\texttt{Positives}\cup G} B_{{c_\infty h}}(x)$.
\end{itemize}
In practice, Algorithm~\ref{alg:grow-alpha-roa} is invoked twice: a first pass with $\texttt{Trim}=\texttt{False}$ grows a tentative region $S^{(0)}$ from local trajectory information alone, and a second pass with $\texttt{Trim}=\texttt{True}$ and initial grid $S^{(0)}$ prunes every box whose certifying trajectory escapes $S^{(0)}$.

\begin{algorithm*}[!htb]
\caption{$\mathrm{Find\text{-}\alpha\text{-RoA}}(R, \varepsilon, \alpha, G_0, \texttt{Trim})$ — Certified $\alpha$-region of attraction}
\label{alg:grow-alpha-roa}

\KwIn{Target rate $\alpha>0$, range $R>0$, $\varepsilon\in(0,R)$, initial grid $G_0=\{(x_i,{h_i})\}$, Boolean $\texttt{Trim}$, max refinements $m$}
\KwOut{Subset of $G_0$ certifying $V(x)=\|x\|$ as an $\varepsilon$-ERLF with rate at least $\alpha$}

Set $G\gets G_0$, $\texttt{Positives}\gets\emptyset$, $\texttt{counter}\gets 0$\;
\While{ $\texttt{counter}\le m-1$ \textbf{and} $G\neq\emptyset$ }{
    Define $\texttt{Region}\gets\bigcup_{(x,{h})\in\texttt{Positives}\cup G} B_{{c_\infty h}}(x)$\tcp*{used only when $\texttt{Trim}$}
    \ForEach{ $(x_i,{h_i})\in G$ }{
        $\alpha_i\gets \texttt{Trim}\ ?\ \alpha_{\max}(x_i,{c_\infty\,h_i};\texttt{Region})\ :\ \alpha_{\max}(x_i,{c_\infty\,h_i})$\tcp*{enforce~\eqref{eq:verify-feasibility} if $\texttt{Trim}$}
        \lIf{ $\alpha_i\ge\alpha$ }{ $\texttt{Positives}\gets\texttt{Positives}\cup\{(x_i,{h_i})\}$ }
    }
    $G\gets\textsc{Split}(G\setminus\texttt{Positives})$; $\texttt{counter}\gets\texttt{counter}+1$\tcp*{refine; advance}
}
\Return $\texttt{Positives}$
\end{algorithm*}

\mypara{Numerical example: Kuramoto oscillators}\label{ssec:experiments2}
We apply Algorithm~\ref{alg:grow-alpha-roa} to the Kuramoto oscillator with uniform coupling. {Here we retain the max norm $\|\cdot\|_\infty$ ($c_\infty=1$): its sublevel sets are boxes, so the covering is tight ($r=h$).} For an $n$-dimensional system with coupling constant $k$, the dynamics are
\begin{align}\label{eq:kuramoto1}
    \dot{\theta}_i = \frac{k}{n}\sum_{j=1}^{n}\sin(\theta_j-\theta_i).
\end{align}
To remove the rotational symmetry, we change variables to $\varphi_i:=\theta_i-\theta_n$, {reducing the state dimension to $d=n-1$}. A direct calculation gives the closed-form bound $L\le 2k(n-1)/n$, which we use throughout without further numerical estimation.

\hl{We first run the three-oscillator case with $k=10$ and target rate $\alpha=1$. Figure~\ref{fig:splits6} shows the phase portrait overlaid with the certified $1$-RoA returned by Algorithm~\ref{alg:grow-alpha-roa} with maximum split count $m=6$. The run uses the two-pass invocation described above: a first pass with $\texttt{Trim}=\texttt{False}$ to grow the candidate region, followed by a second pass with $\texttt{Trim}=\texttt{True}$ to enforce~\eqref{eq:verify-feasibility} against the grown region.}


\begin{figure}
    \;\;\includegraphics[width=0.425\textwidth]{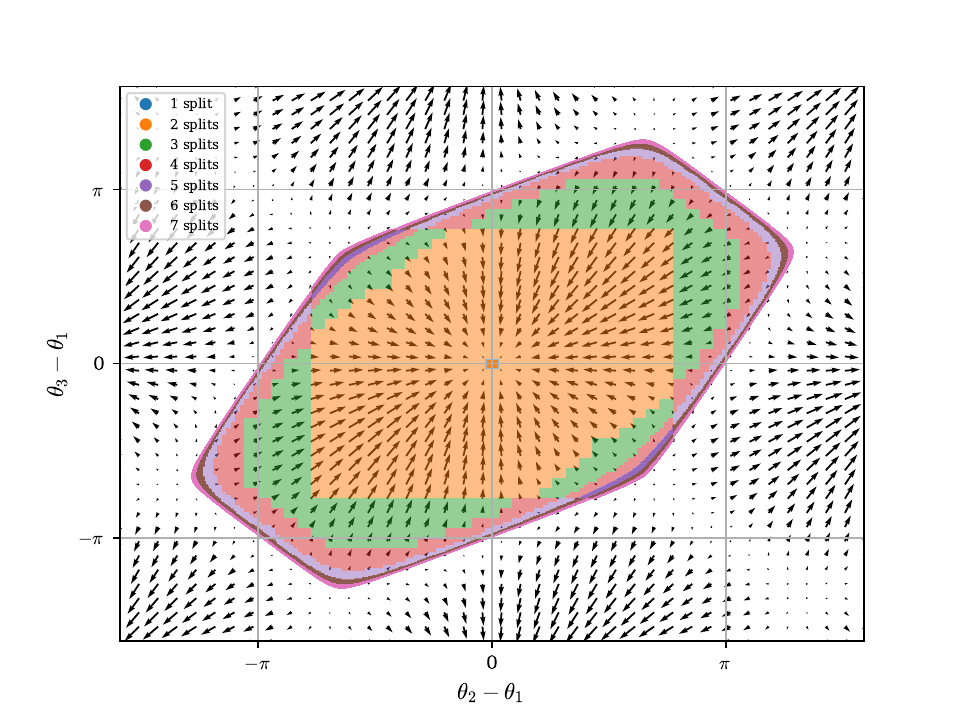}
    \caption{Phase portrait of system~\eqref{eq:kuramoto1} (three oscillators, $k=10$). The shaded region is the certified $1$-RoA returned by Algorithm~\ref{alg:grow-alpha-roa} with maximum split count $m=6$; colors indicate each ball's split depth. We have $\hl{\tau = 1.25}$ and $ \hl{L\simeq 13.33}$.}
    \label{fig:splits6}
\end{figure}

Finally, we examine how the certified region and the runtime scale with the ambient dimension. Holding the system and algorithmic parameters fixed, we vary the state dimension from $2$ to $6$ and the maximum split count from $0$ to $6$, applying $\texttt{Trim}=\texttt{True}$ only at the final pruning pass. Figure~\ref{fig:2dregion} reports the percentage of the basin of attraction certified by Algorithm~\ref{alg:grow-alpha-roa} against the wall-clock runtime. Here the certified set, clipped to $Q_R$, is compared with a Monte Carlo estimate of the basin of the origin within $Q_R$, obtained by integrating $20{,}000$ initial conditions sampled uniformly in $[-\pi,\pi]^d$ up to $T=50$ and declaring convergence when $\|\phi(T,x)\|_\infty < 0.05$.

\begin{figure}
    \;\;\includegraphics[width=0.425\textwidth]{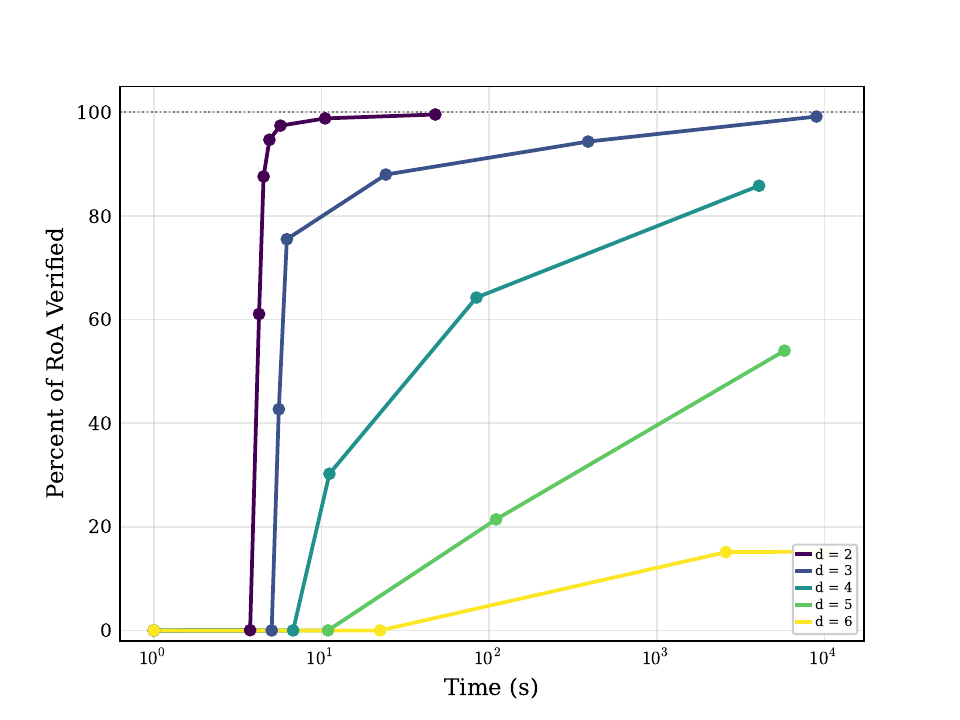}
    \caption{Percentage of the $1$-RoA certified versus wall-clock runtime of Algorithm~\ref{alg:grow-alpha-roa} on system~\eqref{eq:kuramoto1}. Each curve corresponds to a fixed state dimension ($d=2,\dots,6$); points along each curve to maximum split counts $m=0,\dots,6$. Time is on a logarithmic scale.}
    \label{fig:2dregion}
\end{figure}

\section{Conclusions}\label{ct-conclusions}
In this paper, we relaxed the notion of set invariance by introducing set recurrence, and showed that, under mild conditions, recurrent Lyapunov functions suffice to guarantee stability, asymptotic stability, and exponential stability of an equilibrium point. We further established norm-agnostic converse results: under each stability notion, every norm satisfies a slightly weaker version of the corresponding recurrence condition; the weakening is shown to correspond to a notion of practical (asymptotic or exponential) stability. Building on this theory, we developed two parallelizable, trajectory-based algorithms: Algorithm~\ref{alg:region-verification} (best decay rate over a fixed region), \hl{which certifies tighter rates than Sum-of-Squares methods on the bilinear benchmark}, and Algorithm~\ref{alg:grow-alpha-roa} (largest $\alpha$-region of attraction for a fixed rate), \hl{which certifies large regions of attraction on the Kuramoto benchmark}. A complementary sample-complexity bound shows that $O(\log(R/\varepsilon))$ trajectory evaluations suffice, revealing an intrinsic trade-off between certified performance and computational cost.

\section*{Acknowledgments}\label{sec:ack}
The authors thank Eduardo Sontag and Victor Preciado for several insightful comments on earlier versions of this work\red{, and Yue Shen for contributions to the conference version of this manuscript \cite{sspm2023cdc}}.



\bibliographystyle{IEEEtran}
\bibliography{refs.bib}


\begin{IEEEbiography}[{\includegraphics[width=1in,height=1.25in,clip,keepaspectratio]{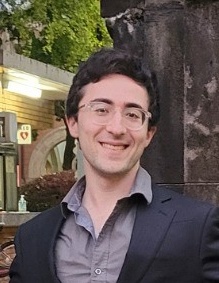}}]
{Roy N. Siegelmann} is a Postdoctoral Associate at the Massachusetts Institute of Technology and a Visiting Postdoctoral Scholar at Harvard University. He received his Ph.D. degree in Applied Mathematics and Statistics from Johns Hopkins University in 2025, along with a master’s degree in Computer Science. He received his B.S. degree in Pure Mathematics with a minor in Computer Science from the University of Massachusetts Amherst.
He was awarded a National Research Service Award (NIH) in 2022 and the Harriet H. Cohen Engineering Fellowship in 2020. His research interests lie at the intersection of control theory, dynamical systems, and machine learning, with applications to large language models, reinforcement learning, and synthetic biology.
\end{IEEEbiography}
\begin{IEEEbiography}[{\includegraphics[width=1in,height
=1.25in,clip,keepaspectratio]{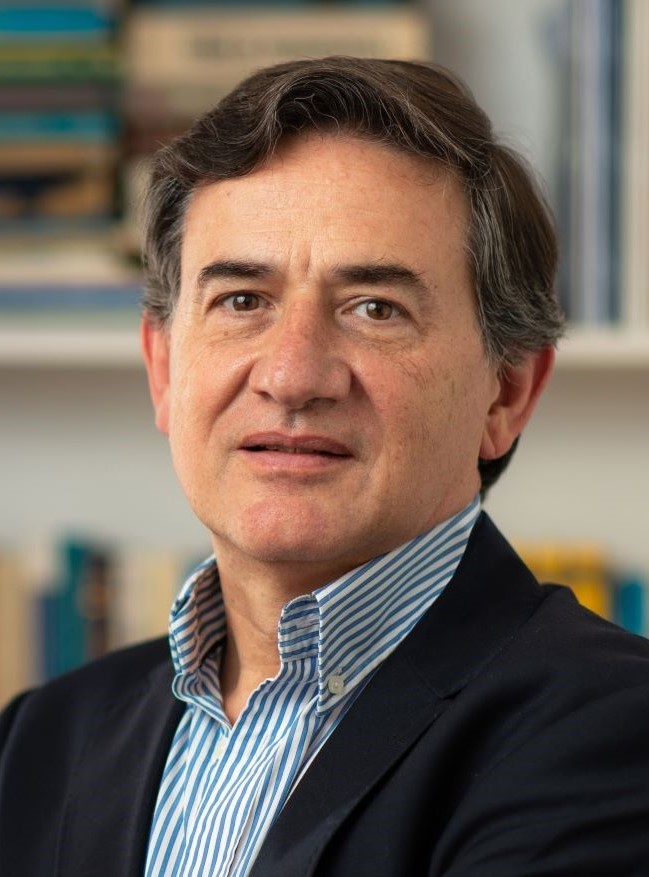}}]{Fernando Paganini} (M'90--SM'05--F'14)
received his degrees in both Electrical Engineering and Mathematics from
Universidad de la Rep\'ublica, Montevideo, Uruguay, in 1990, and his M.S. and PhD degrees in Electrical Engineering from the California
Institute of Technology, Pasadena, in 1992 and 1996 respectively.
His PhD thesis received the 1996 Wilts Prize and the 1996 Clauser
Prize at Caltech. From 1996 to 1997 he was a postdoctoral associate at MIT. Between
1997 and 2005 he was on the faculty the Electrical Engineering
Department at UCLA, reaching the rank of Associate Professor.
Since 2005 he is Professor of Electrical and Telecommunications
Engineering at Universidad ORT Uruguay, and currently Vice-Dean of Research.

Dr. Paganini has received the 1995 O. Hugo Schuck Best Paper
Award, the 1999 Packard Fellowship, the 2004 George S. Axelby Best Paper Award.
He is a member of the Uruguayan National Academy of
Sciences, the Uruguayan National Academy of Engineering, and the Latin American Academy of Sciences. During the pandemic he served in Uruguay as one of three coordinators of the Honorary Scientific Advisory Group on Covid-19, receiving after the Presidency of the Republic Award. He is a Fellow of the IEEE (2014) and a Fellow of IFAC (2023). His research interests are control and networks.

\end{IEEEbiography}
\begin{IEEEbiography}[{\includegraphics[width=1in,height=1.25in,clip,keepaspectratio]{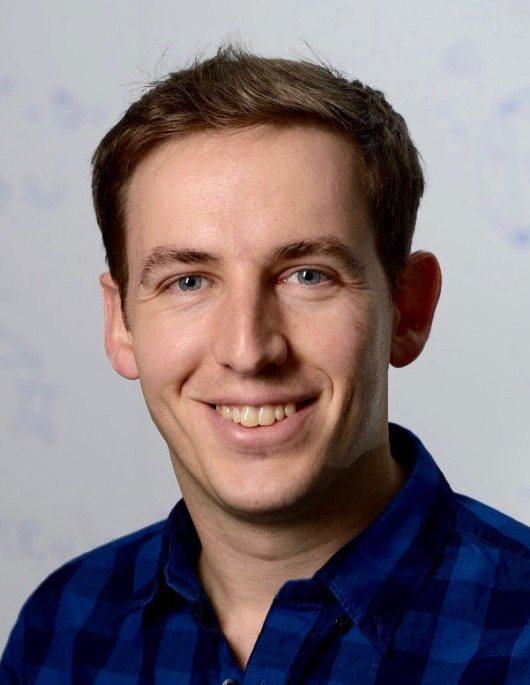}}]
{Enrique Mallada} (S'09-M'13-SM'19) is an Associate Professor of Electrical and Computer Engineering at Johns Hopkins University. Prior to joining Hopkins in 2016, he was a Postdoctoral Fellow in the Center for the Mathematics of Information at Caltech from 2014 to 2016. He received his Ingeniero en Telecomunicaciones degree from Universidad ORT, Uruguay, in 2005 and his Ph.D. degree in Electrical and Computer Engineering from Cornell University in 2014.
Dr. Mallada was awarded
the NSF CAREER award in 2018,
the ECE Director's PhD Thesis Research Award for his dissertation in 2014,
the Center for the Mathematics of Information (CMI) Fellowship from Caltech in 2014,
and the Cornell University Jacobs Fellowship in 2011.
His research interests lie in the areas of control, dynamical systems and optimization, with applications to engineering networks such as power systems and the Internet.
\end{IEEEbiography}

\balance

\ifthenelse{\boolean{with-appendix}}{
\clearpage
\section*{Appendix}
\subsection{Proof of Corollary~\ref{cor:level set connetness}}
\begin{proof}

Let $V$ be the Zubov's function whose existence is guaranteed by Theorem \ref{thm:zubov}. Thus by the definition of $V$, for $c\in(0,1)$, $V_{< c}\subseteq\mathcal{A}(x^*)$. Further from \eqref{eq:dotVF}, it follows that $(\mathcal{L}_fV)(x)<0$, for $x\in V_{< c}\subset \A(x^*)$. Thus, $V_{< c}$ is positive invariant.

To prove the $V_{< c}$ is contractible, we need to provide a continuous mapping $H:[0,1]\times V_{< c} \rightarrow V_{< c}$ such that $H(0,x)=x$ and $H(1,x)=x^*$ for all $x\in V_{< c}$. Similar to~\cite{sontag2013mathematical}, we define $H(s,x):= \phi(\frac{s}{1-s},x)$ for $s<1$, and $H(1,x)\equiv x^*$. Note that $H$ is continuous in $s$ and $x$ for $s<1$ , as in ~\cite{Khalil2002}. 
We are thus left to prove continuity at each $(1,x)$. 
To do so, we take any such $x$ and pick any open neighborhood $\mathcal{V}$ of $H(1,x)=x^*$. By Assumption~\ref{as:roa} as well as the definition of asymptotic stability, it follows that there exist another open neighborhood $\mathcal{W}\subseteq\A(x^*)$ of $x^*$ for which all trajectories starting in $\mathcal{W}$ remain in $\mathcal{V}$, i.e., $\phi(t,x_0)\in\mathcal{V}$ for all $x_0\in\mathcal{W}$ and $t>0$. Given $V_{< c}\subseteq\mathcal{A}(x^*)$, any point $x\in V_{< c}$ satisfies $\phi(T,x)\in\mathcal{W}$ for some $T>0$. This, together with the continuity of $\phi(T,\cdot)$, implies that there is a neighborhood $\mathcal{V}'\subseteq V_{< c}$ of $x$ such that $\phi(T,y)\in\mathcal{W}$ for all $y\in \mathcal{V}'$, which let us conclude:
\begin{align}
    H(s,y)\in \mathcal{V} \quad \text{whenever}\,\, y\in\mathcal{V}'\,\,\text{and}\,\, s>1-\frac{1}{T+1}
\end{align}
and continuity follows since $\mathcal{V}$ could be made arbitrarily small.
\end{proof}


\subsection{Proof of Theorem~\ref{thm:recurrence}}
\begin{proof}
($\implies$):~
If $\R$ is recurrent, then for any $x_0\in\R$, we can construct an infinite sequence $\{x_{n}\}_{n=0}^\infty$ that lies within $\R$, i.e., $\{x_{n}\}_{n=0}^\infty\subset\R$. Precisely, we start from $t_0=0$ that gives solution $x_0:=\phi(0,x_0)\in\R$. Then, given $x_{n}:=\phi(t_n,x_0)\in\R$ and some fixed time interval $\tau>0$, we defined $t_{n+1}$ as the first time since $t_n+\tau$ that the solution $x_{n+1}:=\phi(t_{n+1},x_0)$ lies within $\R$, i.e.,  $\phi(t_{n+1},x_0)\in\R$ and $\phi(t,x_0)\not\in\R$ for all $t\in[t_n+\tau,t_{n+1})$. Note that  {Definition \ref{defn:recurrent}} ensures there exist such a $t_{n+1}$ and $x_{n+1}$. 

Now since $\R$ is compact, by Bolzano-Weierstrass theorem, $\{x_{n}\}_{n=0}^\infty$ must have a sub-sequence  $\{x_{n_i}\}_{i=1}^\infty$ that converges to an accumulation point $\bar x\in \R$. It follow then by the definition of $\omega$-limit sets (Definition \ref{defn:omega-limit-set}) that $\bar x= \lim_{i\rightarrow \infty}x_{n_i}\in \Omega(f)\cap\R\not=\emptyset$. Thus, we have that $x_0\in\Omega(f)\cap\R$. Finally, since $x_0$ was chosen arbitrarily within $\R$, it follows that $\R\subset\A(\Omega(f)\cap\R)$.

\noindent
($\Longleftarrow$):~
By assumption $\Omega(f)\cap\R\subset \opint\R$ and $\R \subset \A(\Omega(f)\cap\R)$. Therefore, if $x_0\in\R$, then $x_0\in\A(\Omega(f)\cap\R)$ and it follows that $\phi(t,x_0)$ converges to $\Omega(f)\cap\R\subset\R$. Therefore for all $x_0\in\R$, since $\Omega(f)\cap\R\subset \opint\R$, {it follows from the continuity of $\phi$ that there always exists some time $t>0$ such that $\phi(t,x_0)\in\R$. Thus $\R$ is recurrent.}
\end{proof}

\subsection{Proof of Corollary~\ref{cor:roa-subet}}
\begin{proof}
($\implies$):~
By assumption $\R$ is compact, $\partial\R\cap\Omega(f)=\emptyset$, Theorem \ref{thm:recurrence} implies that if $\R$ is recurrent then $\Omega(f)\cap\R\not=\emptyset$ and $\R\subset\A(\Omega(f)\cap\R)$. Since all the equilibrium points inside $\Omega(f)$ are hyperbolic, the regions of attraction of the unstable ones are not full dimensional~\cite{chiang1988}, and therefore $\A(\Omega(f)\cap\R)\backslash\A(x^*)$ is not full dimensional. It follows that the set $\R\backslash\A(x^*)\subset\A(\Omega(f)\cap\R)\backslash\A(x^*)$ is also not full dimensional. Together with the fact that ROA $\mathcal{A}(x^*)$ is an open contractible set, one can conclude that $\R\backslash\A(x^*)\subseteq \partial \R$. Otherwise, there exists a point $x\in\R\backslash\A(x^*)$ satisfying $x\in \opint (\R)$, which contradict with $\R\backslash\A(x^*)$ not full dimensional.

Now since $(\Omega(f)\cap \R)\backslash x^* \subseteq \R \backslash \A(x^*)\subseteq \partial \R$ contradict with $\partial\R\cap\Omega(f)=\emptyset$ if $(\Omega(f)\cap \R)\backslash x^* $ is non-empty, we can further conclude that  $\Omega(f)\cap \R =\{x^*\}$ given $\Omega(f)\cap\R\not=\emptyset$. And the other conclusion follows from $\R\subset\A(\Omega(f)\cap\R)=\A(x^*)$.

($\Longleftarrow$):~
This direction is trivial given Theorem~\ref{thm:recurrence}.
\end{proof}

\subsection{Proof of Theorem~\ref{thm:bounded k}}
\begin{proof}
The proof of the theorem relies on Zubov's existence criterion stated in Theorem \ref{thm:zubov}. Given $\R$, let us now define
\begin{align}
    \underline{c} := \min_{x\in\partial\R} V(x), \quad \overline{c} := \max_{x\in\partial\R} V(x), \\
    \text{and}\quad a := \max_{x\in C} \nabla V(x)^T f(x),
\end{align}
where $C=\{x\in\mathbb{R}^d: \underline{c}\leq V(x)\leq\overline{c}\}$ is compact. 

We first argue that $V_{\leq\underline{c}}:=\{x:V(x)\le \underline{c}\}\subseteq \R$. Let $\underline{x}$ be the point in $\partial\R$ that achieves the minimum, i.e, $V(\underline{x})=\underline{c}$, and let $\R'$ be the connected component of $\R$ containing $\underline{x}$. Note that $x^*\in\opint \R$ must be contained in $\R'$, since otherwise, the trajectory $\phi(t,\underline{x})$, which strictly decreases $V$ must eventually find a point $x'\in\partial R$ with $V(x')<\underline{c}$; which contradicts the  definition of $\underline{c}$. Thus, $x^*\in\R'\subseteq\R$. 

Suppose then that $V_{\leq \underline{c}}\not\subseteq\R'\subseteq\R$, for any point $\tilde x\in V_{\leq \underline{c}}\backslash\R'$, $V(\phi(t,\tilde x))<\underline{c}$, for $t>0$, and $\lim_{t\rightarrow\infty}\phi(t,\tilde x)=x^*$. Thus there exists $\tilde t>0$ s.t. $V(\phi(t,\tilde x))<\underline{c}$ and $\tilde x\in\partial \R$; contradiction. It follow then that $V_{\leq\underline{c}}\subseteq \R' \subseteq \R$.

Similarly, since $V_{\leq\overline{c}}$ contains every point in the boundary of $\R$, there cannot be any point in $x\in\R$ with $V(x)>\overline{c}$. 
We therefore get that the following inclusions must hold:
\begin{equation}\label{eq:inclusion}
V_{\leq\underline{c}}\subseteq \R \subseteq V_{\leq\overline{c}}.
\end{equation}

Finally, by \eqref{eq:inclusion}, for any point $x\in \R$ we must have $V(x)\leq \overline{c}$. Since the
time derivative of $V(x)$ is at most $a<0$, it follows that {after $t\ge \Bar{\tau}:=\frac{\underline{c}-\overline{c}}{a}$ the Lyapunov value $V(\phi(t,x))\leq \underline{c}$, which implies that $\phi(t,x)\in\R$ and result follows.}
\end{proof}

\subsection{Proof of Theorem~\ref{thm:recurrent subset}}
The proof of Theorem \ref{thm:recurrent subset} is analogous to Theorem \ref{thm:bounded k} and omitted due to space constraints.

\begin{figure}
    \center
    \includegraphics[width=0.33\textwidth]{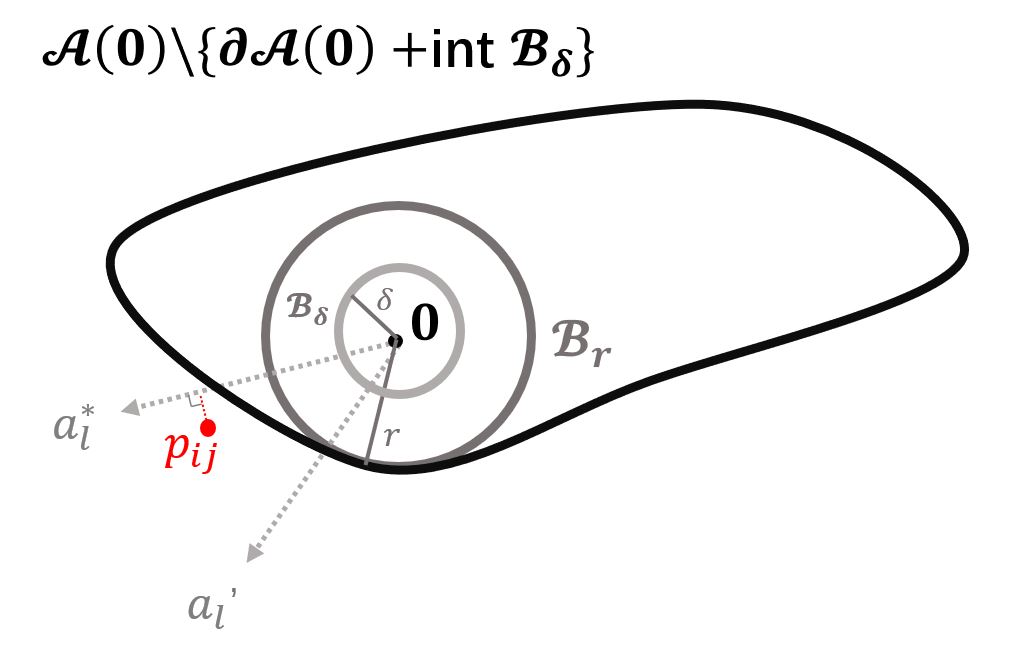}
    \caption{An illustration of the proof of Theorem~\ref{thm: inclusion}. In particular, given arbitrary point $p_{ij} \not \in \A(0)\backslash\{\partial\A(0)+\opint\mathcal{B}_{\delta}\}$, in the sphere case, it follows that $\norm{p_{ij}}_2\ge r$. And in the polyhedron case, the closest projection  $\max_{l\in \{1,...,n\}} {a_l^T p_{ij}} \ge r/2$ under Assumption~\eqref{eq:r net} that every pair of exploration directions are close enough.}
    \label{fig:inclu thm}
\end{figure}

\subsection{Proof of Theorem~\ref{thm:k-recurrent subset}}
\begin{proof}
Given Theorem~\ref{thm:recurrent subset}, this result  follows directly from $\phi(t,x)\in \R$ for all $x\in\R$ when $t\ge\Bar{\tau}(\delta)$.
\end{proof}

\subsection{Proof of Theorem~\ref{thm: inclusion}}
\begin{proof}
Given an arbitrary counter-example $p_{ij}$ w.r.t $\hat{\mathcal{S}}^{(i)}\supseteq\mathcal{B}_\delta$, it follows that $p_{ij} \not \in \A(0)\backslash\{\partial\A(0)+\opint\mathcal{B}_{\delta}\}$ by Theorem~\ref{thm:k-recurrent subset}; since otherwise, $p_{ij}$ would generate a $k$-recurrent trajectory. Then, as illustrated in Figure \ref{fig:inclu thm}, $\norm{p_{ij}}_2\ge r$. Further, let $\mathcal{B}_r:= \{x|\norm{x}_2 \le r\}\subseteq \mathcal{A}(0)\backslash\{\partial \mathcal{A}(0)+\opint\mathcal{B}_{\delta}\}$. 

We now reason differently depending on the type of approximation.
\vspace{-3ex}
\begin{description}[align=left,style=nextline,leftmargin=*,labelsep=\parindent,font=\normalfont]
\item \emph{(Sphere case):}~
It then follows from $\norm{p_{ij}}_2\ge r$ that whenever $\varepsilon\le r-\delta$, the update leads to $\,b^{(i+1)} = \norm{p_{ij}}_2-\varepsilon \ge r-\varepsilon \ge \delta$.\vspace{-2.5ex}
 \item \emph{(Polyhedron case):}~
 It follows from~\eqref{eq:r net}, that for any point $p' \not\in \mathcal{B}_r$, we have $\max_{l\in \{1,...,n\}}  a_l^T p'\geq \|p'\|\cos \left(\frac{2}{3}\pi\right)\geq \frac{r}{2}.$
 Therefore, since by definition of $\mathcal B_r$, $p_{ij}\not\in \mathcal B_r$ we conclude then that
 $b^{(i+1)}_{l^*} =a_{l^*}^T p_{ij}-\varepsilon\ge \frac{r}{2} - \varepsilon\geq \delta.$
\end{description}
Together with the fact that $\hat{\mathcal{S}}^{(0)}\supseteq \mathcal{B}_\delta$, result follows.
\end{proof}

\subsection{Proof of Theorem~\ref{thm:counter example upper bound}}
\begin{proof}
 Note that once a counter-example is encountered, we decrease the radius constraint (sphere case) or on one of the exploration directions (polyhedron case) by at least $\varepsilon$. Therefore, $\hat{\mathcal{S}}^{(i)}\in \mathcal{F}_c$ for all $i\in\{1,2,...\}$. And for any fixed $k$, our method can find at most $\frac{c}{\varepsilon}$ counter-examples with the sphere approximation and $n\frac{c}{\varepsilon}$ counter-examples with the polyhedron approximation without failing. Since it takes at most $\log_2{\Bar{k}}$ updates on $k$ to find some $k\geq \bar k$ using the doubling method, result follows.
\end{proof}


\subsection{Proof of Lemma~\ref{lem:scountervol}}
\begin{proof}
We will proof this statement by contrapositive, i.e., we will show $\vol(\mathcal{S}_\text{\emph{counter}})=0$ implies $\opint\mathcal{S}\subseteq\mathcal{A}(0)$.

We first argue that if $\vol(\mathcal{S}_\text{\emph{counter}})=0$ then $\mathcal{S}\backslash \A(0)\subseteq\partial\mathcal{S}$. To see this, we can form a contradiction. Assume $\mathcal{S}\backslash \A(0)\not\subseteq\partial\mathcal{S}$ and recall $\mathcal{A}(0)$ is an open contractible set. It follows then that there exist a full dimensional set $\mathcal{V}\subseteq\mathcal{S}\backslash\A(0)$. Now note that the regions of attraction of the unstable hyperbolic equilibrium points are not full dimensional. There is then a point $p\in \opint\mathcal{V}$ satisfying $p\in \mathcal{S}_\text{\emph{counter}}$. By the continuity of $\phi(\tau,\cdot)$, we further have an open neighborhood $\mathcal{V}'$ of $p$ satisfying $\vol(\mathcal{V}')>0$ and $\mathcal{V}'\subseteq \mathcal{S}_\text{\emph{counter}}$, which contradict with $\vol(\mathcal{S}_\text{\emph{counter}})=0$.

In summary, we have $\mathcal{S}\backslash \A(0)\subseteq\partial\mathcal{S}$, which imples that $\{\mathcal{S}\backslash \A(0)\}\backslash\{\partial \mathcal{S}\}=\emptyset$. Therefore, we have $\opint\mathcal{S} \subseteq \mathcal{A}(0)$ and result follows.
\end{proof}

\subsection{Proof of Lemma~\ref{lem:counter example as}}
\begin{proof}
Note that we have $\mathcal{S}_{\text{counter}}\subseteq\mathcal{S}$ and $\vol(\mathcal{S}_\text{counter})>0$ by Lemma~\ref{lem:scountervol}. Then, denoting the counter-example ratio as $\rho:=\vol(\mathcal{S}_\text{counter})/\vol(\mathcal{S})$, one can conclude $0<\rho\le1$ and 
\begin{align}
    \lim_{m\rightarrow\infty}\mathbb{P}(X_0=...=X_m=1)=\lim_{m\rightarrow\infty}(1-\rho)^{m} = 0.
\end{align}
\end{proof}

\subsection{Proof of Theorem~\ref{thm: Sas}}
\begin{proof}
  Suppose that at any given iteration $i$ the set $\opint \hat{\mathcal{S}}^{(i)}\not\subseteq\mathcal{A}(0)$. Then it follows from Lemma \ref{lem:counter example as} that a counter-example is eventually found almost surely, and a new set $\hat{\mathcal{S}}^{(i+1)}$ is obtained.
  Also Theorem~\ref{thm:counter example upper bound} implies the total number of such transitions is finite, since $\hat{\mathcal{S}}^{(0)}\in \mathcal{F}_c$ and $\mathcal{B}_\delta \subseteq\hat{\mathcal{S}}^{(0)}$.
  
  Now let $\hat{\mathcal{S}}^*$ denote the last updated approximation. Note that since there are not further updates to  $\hat{\mathcal{S}}^*$  with probability one, this implies that $\vol(\hat{\mathcal{S}}^*_{\text{counter}})=0$. We argue then that $\opint \hat{\mathcal{S}}^{*}\subseteq\mathcal{A}(0)$, since otherwise $\vol(\hat{\mathcal{S}}^*_{\text{counter}})>0$, which contradicts the fact that $\hat{\mathcal{S}}^{*}$ is the last iteration. Finally, $\opint \hat{\mathcal{S}}^*$ is non-empty since Theorem~\ref{thm: inclusion} implies $\hat{\mathcal{S}}^*\supseteq\mathcal{B}_\delta$. 
\end{proof}

\subsection{Proof of Theorem~\ref{thm: multi center}}
  \begin{proof}
  By definition $\hat{\mathcal{S}}_{\text{multi}}^{(i)}\supseteq\hat{\mathcal{S}}_1^{(i)}$ for all $i\in\mathbb{N}^+$, Theorem~\ref{thm: inclusion} therefore implies $\hat{\mathcal{S}}_{\text{multi}}^{(i)}\supseteq\hat{\mathcal{S}}_1^{(i)}\supseteq\mathcal{B}_\delta$ under \eqref{eq:epsilon}. The  bound on the total number of counter examples follows from Theorem~\ref{thm:counter example upper bound}. Since every additional approximation $\hat{S}_q^{(i)}\in \mathcal{F}_c$ for all $q\in\{1,...,h\}$ and iteration $i\in\{1,...\}$. Finally, by generalizing  Lemma ~\ref{lem:scountervol}, Lemma ~\ref{lem:counter example as} and Theorem~\ref{thm: Sas} to the scope of $\mathcal{F}_c^h$, the last statement follows.
  \end{proof}
}{}
\end{document}